\documentclass[12pt]{article}
\usepackage{amsmath,amsthm}
\usepackage{amssymb,latexsym}
\usepackage{enumerate}

\usepackage{amscd}
\usepackage{amsfonts}
\usepackage{graphicx}
\usepackage{amsmath}
\usepackage{amssymb}
\usepackage[margin=1.5in]{geometry}%
\providecommand{\U}[1]{\protect\rule{.1in}{.1in}}
\newtheorem{theorem}{Theorem}
\newtheorem{acknowledgement}[theorem]{Acknowledgement}

\newtheorem{definition}[theorem]{Definition}
\newtheorem{example}[theorem]{Example}

\newtheorem{lemma}[theorem]{Lemma}

\newtheorem{proposition}[theorem]{Proposition}
\newtheorem{remark}[theorem]{Remark}

\numberwithin{equation}{section}
\begin{document}
\title{Admissibility for Quasiregular Representations of Exponential Solvable Lie Groups}

\author{Vignon Oussa\\
Bridgewater State University\\
Bridgewater, MA 02325 U.S.A.\\
E-mail: vignon.oussa@bridgew.edu}

\date{}

\maketitle


\renewcommand{\thefootnote}{}



\renewcommand{\thefootnote}{\arabic{footnote}}
\setcounter{footnote}{0}


\begin{abstract}
Let $N$ be a simply connected, connected non-commutative nilpotent Lie group
with Lie algebra $\mathfrak{n}$ of dimension $n.$ Let $H$ be a subgroup of the
automorphism group of $N.$ Assume that $H$
is a commutative, simply connected, connected Lie group with Lie algebra $\mathfrak{h}.$ Furthermore, let us
assume that the linear adjoint action of $\mathfrak{h}$ on $\mathfrak{n}$ is
diagonalizable with non-purely imaginary eigenvalues. Let $\tau=\mathrm{Ind}%
_{H}^{N\rtimes H} 1$. We obtain an explicit direct integral decomposition for
$\tau$, including a description of the spectrum as a sub-manifold of
$(\mathfrak{n}+\mathfrak{h})^{\ast}$, a formula for the multiplicity
function of the unitary irreducible representations occurring in the direct
integral, and a precise intertwining operator. Finally, we completely settle
the admissibility question of $\tau$. In fact, we show that if $G=N\rtimes H$
is unimodular, then $\tau$ is never admissible, and if $G$ is nonunimodular,
$\tau$ is admissible if and only if the intersection of $H$ and the center of
$G$ is equal to the identity of the group. The motivation of this work is to contribute to the general theory of admissibility, and also to shed some light on the existence of continuous wavelets on non-commutative and connected nilpotent Lie groups. 
\end{abstract}
\maketitle







\begin{acknowledgement}
Thanks go to Professor Bradley Currey who introduced me to the theory of admissibility. Without his support, and guidance this work would have never been possible. I also thank my wife Lindsay and my three children Senami, Kemi and Donah for their support.
\end{acknowledgement}

\section{Introduction}

Let $\pi$ be a unitary representation of a locally compact group $X,$ acting
in some Hilbert space $\mathcal{H}.$ We say that $\pi$ is admissible, if and only if
there exists some function $\phi\in \mathcal{H}$ such that the operator $W_{\phi}$
defines an isometry on $\mathcal{H},$ and $W_{\phi}:\mathcal{H}\rightarrow L^{2}\left(  X\right)
,$ $W_{\phi}\psi\left(  x\right)  =\left\langle \psi,\pi\left(  x\right)
\phi\right\rangle .$ For continuous wavelets on the real line, the admissibility of the quasiregular representation $\mathrm{Ind}_{\left(  0,\infty\right)  }%
^{\mathbb{R}\rtimes\left(  0,\infty\right)  }1$ of the `ax+b' group which is a unitary representation acting in $L^2(\mathbb{R})$ leads to the well-known Calderon condition. 

Given any locally compact group, a great deal is already known about the
admissibility of its left regular representation \cite {Fuhr}. For example, it is known
that the left regular representation of the `ax+b' group is admissible. The
left regular representation of $\mathbb{R}\rtimes\left(  0,\infty\right)  $
admits a decomposition into a direct sum of two unitary irreducible
representations acting in $L^{2}((0,\infty))$; each with infinite
multiplicities. Thus, the Plancherel measure of this affine group, is supported
on $2$ points. It is also known that the quasiregular representation
$\mathrm{Ind}_{\left(  0,\infty\right)  }^{\mathbb{R}\rtimes\left(
0,\infty\right)  }1$ is unitarily equivalent with a subrepresentation of the
left regular representation, and thus, is admissible.

Several authors have studied the admissibility of various representations; see
\cite{Ali}, and also \cite{Guido}, where Guido Weiss and his collaborators
obtained an almost complete characterization of groups of the type $H\leq
GL(n,\mathbb{R})$ for which the quasiregular representation $\tau
=\mathrm{Ind}_{H}^{\mathbb{R}^{n}\rtimes H}1$ is admissible. It is known that
if $\tau$ is admissible then the stabilizer subgroup of the action of $H$ on
characters belonging to the unitary dual of $\mathbb{R}^{n}$ must be compact
almost everywhere. However, this condition is not sufficient to guarantee the
admissibility of $\tau$. In \cite{Fuhr3}, a complete characterization of
dilation groups $H\leq GL(n,\mathbb{R})$ is given. On non-commutative
nilpotent domains, Liu and Peng answered the question for $\tau=\mathrm{Ind}%
_{H}^{N\rtimes H}1$, where $N$ is the Heisenberg group, and $H$ is a
1-parameter dilation group. They have also constructed some explicit continuous
wavelets on the Heisenberg group (see \cite{Liu}). In 2007, Currey considered
$\tau=\mathrm{Ind}_{H}^{N\rtimes H}1$, where $N$ is a connected, simply
connected non commutative nilpotent Lie group, and $H$ is a commutative,
connected, simply connected Lie group such that $G=N\rtimes H$ is completely
solvable and $\mathbb{R}$-split. He settled the admissibility question for
$\tau$ under the restriction that the stabilizer subgroup inside $H$ is
trivial, and he also gave some explicit construction of some continuous
wavelets (see \cite{Currey 4}). However, he did not address the case, where
the stabilizer of the action of $H$ on the unitary dual of $N$ is non trivial;
leaving this problem open. In 2011, we provided some answers for the
admissibility of monomial representations for completely solvable exponential
Lie groups in \cite{CV}. We now know that when $N$ is not commutative, the stabilizer of the action of $H$ on the dual of $N$ does not have to be trivial in order for $\tau$ to be admissible. We remark that such fact is always false if $N$ is commutative. Also, we were recently informed that new results on the subject of admissibility were obtained by Cordero, and Tabacco in \cite{Cordero}, and Filippo De Mari and Ernesto De Vito in \cite{Mari} for a different class of groups.

The purpose of this paper is to extend the results of Currey \cite{Currey}. Firstly, we make no assumption that the little group inside $H$ is trivial. Secondly, the class of groups considered in this paper is larger than the class considered by Currey. This class of groups also contains exponential solvable Lie groups which
are not completely solvable. We consider the situation where the
action of $\mathfrak{h}$ on $\mathfrak{n}$ has roots of the type $\alpha+i\beta$, with $\alpha\not=0.$ Let us be more precise. Let $N$ be a simply connected, connected non-commutative nilpotent Lie group
with real Lie algebra $\mathfrak{n}.$ Let $H$ be a subgroup of the
automorphism group of $N$, which we denote by $\mathrm{Aut}\left(  N\right)
.$ Assume that $H$ is isomorphic to $%
\mathbb{R}
^{r}$ with Lie algebra $\mathfrak{h}.$ Furthermore, let us assume that the
linear adjoint action of $\mathfrak{h}$ on $\mathfrak{n}$ is diagonalizable
with non-purely imaginary complex eigenvalues. We form the semi-direct
product Lie group $G=N\rtimes H$ such that $G$ is an exponential solvable Lie
group with Lie algebra $\mathfrak{g}.$ More precisely, there exist basis
elements such that $\mathfrak{h}=\mathbb{R}A_{1}\oplus\cdots\oplus
\mathfrak{\mathbb{R}}A_{r},$ and basis elements $Z_{i}$ for the
complexification of $\mathfrak{n}$ such that $Z_{i}$ are eigenvectors for the
linear operator $adA_{k},k=1,\cdots,r.$ Furthermore, we have $ad A_{k}%
Z_{j}=\left[  A_{k},Z_{j}\right]  =\gamma_{j}\left(  A_{k}\right)  Z_{j}$ with
weight $\gamma_{j}\left(  A_{k}\right)  =\lambda\left(  A_{k}\right)  \left(
1+i\alpha_{j}\right)  ,$ $\lambda\in\mathfrak{h}^{\ast},\text{ a real-valued
linear functional, and }\alpha_{j}\in\mathbb{R}.$ $G$ is an exponential
solvable Lie group, and is therefore type $\mathrm{I}$. We define the action of $H$ on $N$ multiplicatively, and
the multiplication law for $G$ is obtained as follows: $\left(  n,h\right)
\left(  n^{\prime},h^{\prime}\right)  =\left(  nh\cdot n^{\prime},hh^{\prime
}\right)  .$ The Haar measure of $G$ is $\left\vert \det Ad\left(  h\right)
\right\vert ^{-1}dndh$, where $dn,dh$ are the canonical Haar measures on $N,H$
respectively. We will denote by $L$ the left regular representation of $G$
acting in $L^{2}\left(  G\right)  .$ We consider the quasiregular
representation $\tau=\mathrm{Ind}_{H}^{G}\left(  1\right)  $ acting in
$L^{2}\left(  N\right)  $ as follows%
\begin{align*}
\tau\left(  n,1\right)  f\left(  m\right)   &  =f\left(  n^{-1}m\right)  \\
\tau\left(  1,h\right)  f\left(  m\right)   &  =\left\vert \det\left(
Ad\left(  h\right)  \right)  \right\vert ^{-1/2}f\left(  h^{-1}m\right)  .
\end{align*}

In this paper, mainly motivated by the admissibility question of $\tau$, we
aim to obtain an explicit decomposition of $\tau$, including a precise
description of its spectrum, an explicit formula for the multiplicity
function, the measure occurring in the decomposition of $\tau$, and finally,
we completely settle the admissibility question for $\tau.$ Here is the main
result of our paper.

\begin{theorem}
Let $N$ be a simply connected, connected non commutative nilpotent Lie group
with Lie algebra $\mathfrak{n}$ of dimension $n.$ Let $H$ be a subgroup of the
automorphism group of $N$. Assume that $H$
is a commutative simply connected, connected Lie group with Lie algebra $\mathfrak{h}.$ Furthermore, let us
assume that the linear adjoint action of $\mathfrak{h}$ on $\mathfrak{n}$ is
diagonalizable with non-purely imaginary eigenvalues such that $N\rtimes H$ is
an exponential solvable Lie group. Let $\tau=\mathrm{Ind}_{H}^{N\rtimes H}1$.

\begin{enumerate}
\item Assume that $\dim(H\cap Z(G))=0$. $\tau$ is admissible if and only if
$N\rtimes H$ is nonunimodular.

\item Assuming that $\dim(H\cap Z(G)) \not =0$, $\tau$ is never admissible.
\end{enumerate}
\end{theorem}

\section{Preliminaries}

We recall that the coadjoint action of $G$ on $\mathfrak{g}^{\ast}$ is simply
the dual of the adjoint action, and is also defined multiplicatively as
$g\cdot l\left(  X\right)  =l\left(  Ad_{g^{-1}}X\right)  ,g\in G,X\in
\mathfrak{g}^{\ast}.$ In this paper, the group $G$ always stands for $N\rtimes H$ as described earlier.

\begin{definition}
Given 2 representations $\pi,\theta$ of $G$ acting in the Hilbert spaces
$\mathcal{H}_{\pi},\mathcal{\mathcal{H}}_{\theta}$ respectively, if there exists a bounded linear operator
$T:\mathcal{H}_{\pi}\rightarrow\mathcal{H}_{\theta}$ such that $\theta\left(  x\right)
T=T\pi\left(  x\right)  $ for all $x\in G,$ we say $T$ intertwines $\pi$ with
$\theta.$ If $T$ is a unitary operator, then we say the representations are
unitarily equivalent, we write $\pi\simeq\theta,$ and $[\pi]=[\theta]$.
\end{definition}

\begin{lemma} \label{lemma1} Let $L$ be the left regular representation of $G$ acting in
$L^{2}(G)$. $L$ is admissible if and only if $G$ is nonunimodular.
\end{lemma}

Lemma \ref{lemma1} was proved in more general terms by Hartmut F\"uhr
in Theorem 4.23 \cite{Fuhr}. In fact, the general statement of his proof only assumes that
$G$ is type I and connected.

\begin{lemma}\label{lem5}
If $G$ is nonunimodular, $\tau$ is admissible if and only if $\tau$ is
equivalent with a subrepresentation of $L.$
\end{lemma}

\begin{lemma}
\label{irr} Let $\pi,\rho$ be two type I unitary representations of $G$ with
the following direct integral decomposition. $\pi\simeq\int_{\widehat{G}%
}^{\oplus} \sigma\otimes1_{\mathbb{C}^{m_{\pi}}} d\mu(\sigma)$, and
$\rho\simeq\int_{\widehat{G}}^{\oplus} \sigma\otimes1_{\mathbb{C}^{m^{\prime
}_{\rho}}} d\mu^{\prime}\sigma.$ $\pi$ is equivalent with a subrepresentation
of $\rho$ if and only if $\mu$ is absolutely continuous with $\mu^{\prime}$
and $m_{\pi}\leq m^{\prime}_{\rho}$ $\mu$ a.e.
\end{lemma}

A clear explanation of Lemma \ref{lem5} and Lemma \ref{irr} is given on Page 126 of the Monograph \cite{Fuhr}. The following theorem is due to Lipsman, and the proof is in Theorem $7.1$
in \cite{Lipsman}.

\begin{lemma}
\label{Ronald} Let $G=N\rtimes H$ be a semi-direct product of locally compact
groups, $N$ normal and type I. Let $\gamma\in\widehat{N},$ $H_{\gamma}$ the
stability group. Let $\widetilde{\gamma}$ be any extension of $\gamma$ to
$H_{\gamma}.$ Suppose that $N$ is unimodular, $\widehat{N}/H$ is countably
separated and $\widetilde{\gamma}$ is a type I representation for $\mu_{N}$
almost everywhere $\gamma\in\widehat{N}$. Let
\[
\widetilde{\gamma}\simeq\int_{\widehat{H_{\gamma}}}^{\oplus}n_{\gamma}\left(
\sigma\right)  \sigma d\mu_{\gamma}\left(  \sigma\right)
\]
be the unique direct integral decomposition of $\widetilde{\gamma}.$ Then
\[
\mathrm{Ind}_{H}^{G}1\simeq\int_{\widehat{N}/H}^{\oplus}\int_{\widehat
{H_{\gamma}}}^{\oplus}\pi_{\gamma,\sigma}\otimes1_{%
\mathbb{C}
^{n_{\gamma}\left(  \sigma\right)  }}d\mu_{\gamma}\left(  \sigma\right)
d\overset{\cdot}{\mu}_{N}\left(  \gamma\right)  ,
\]
where $\overset{\cdot}{\mu}_{N}$is the push-forward of the Plancherel measure
on $\mu_{N}$ on $\widehat{N}.$
\end{lemma}

It is now clear that in order to settle the admissibility question, it is
natural to compare both representations. Being that $G$ is a type I group,
there exist unique direct integral decompositions for both $L$ and $\tau.$
Since both representations use the same family of unitary irreducible
representations in their direct integral decomposition, in order to compare
both representations, it is important to obtain the direct integral
decompositions for both $L$ and $\tau$, and to check for the containment of
$\tau$ inside $L$. In order to have a complete picture of the results in Lemma
\ref{Ronald}, we will need the following.

\begin{enumerate}
\item A precise description of the spectrum of the quasiregular representation.

\item The multiplicity function of the irreducible representations occurring
in the decomposition of the quasiregular representation.

\item A description of the push-forward of the Plancherel measure of $N$.
\end{enumerate}

Our approach here, will rely on the orbit method, and we will construct a
smooth orbital cross-section to parametrize the dual of the group $G.$ 

\section{Orbital Parameters}

In this section, we will introduce the reader to the theory developed by Currey, and Arnal, and Dali in \cite{Didier} for the construction of
cross-sections for coadjoint orbits in $\mathfrak{g}^{\ast}$, where
$\mathfrak{g}$ is any $n$-dimensional real exponential solvable Lie algebra
with Lie group $G$. First, we consider a complexification of the Lie algebra
$\mathfrak{g}$ which we denote here by $\mathfrak{c}=\mathfrak{g}_{\mathbb{C}%
}$. Let us be more precise. We begin by fixing an ordered basis $\left\{  Z_{1},\cdots,Z_{n}\right\}
$ for the Lie algebra $\mathfrak{c}$, where $Z_{i}=\operatorname{Re}%
Z_{i}+i\operatorname{Im}Z_{i}$, $\operatorname{Re}\left(  Z_{i}\right)  $, and
$\operatorname{Im}\left(  Z_{i}\right)  $ belong to $\mathfrak{g}$ such that
the following conditions are satisfied:

\begin{enumerate}
\item For each $k\in\{1,\cdots,n\},$ $\mathfrak{c}_{k}=\mathbb{C}\text{-}\mathrm{span}%
\left\{  Z_{1},Z_{2},\cdots,Z_{k}\right\}  $ is an ideal.

\item If $\mathfrak{c}_{j}\neq\overline{\mathfrak{c}_{j}}$ then
$\mathfrak{c_{j+1}}=\overline{\mathfrak{c}_{j+1}}$ and $Z_{j+1}=\overline
{Z_{j}}$.

\item If $\mathfrak{c}_{j}=\overline{\mathfrak{c}_{j}}$ and $\mathfrak{c}
_{j-1}=$ $\overline{\mathfrak{c}_{j-1}}$ then $Z_{j}\in\mathfrak{g}$.

\item For any $A\in$ $\log H,\left[  A,Z_{j}\right]  =\gamma_{j}
(A)Z_{j}\operatorname{mod}\mathfrak{c}_{j-1}$ with weight $$\gamma_j (A_k)=\lambda(A_k) (1+i \alpha_j).$$  $\lambda\in\mathfrak{h}^{\ast},\text{ is a real-valued
linear functional, and }\alpha_{j}\in\mathbb{R}.$
\end{enumerate}

Such basis is called an \textbf{adaptable basis}. We recall the procedure
described in \cite{Didier}. For any $l\in\mathfrak{g}%
^{\ast}$, we define for any subset $\mathfrak{s}$ of $\mathfrak{c}$,
$\mathfrak{s}^{l}=\{Z\in\mathfrak{c\ }:$ $l\left(  \left[  \mathfrak{s,}%
Z\right]  \right)  =0\}$ and $\mathfrak{s}(l)=\mathfrak{s}^{l}\cap
\mathfrak{s}$. Also, we define
\begin{align*}
i_{1}(l) =\min\left\{  j:\mathfrak{c}_{j}\text{ }\not \subset \text{
}\mathfrak{c}(l)\right\} ,\\
\mathfrak{h}_{1}(l)  =\mathfrak{c}_{i_{1}}^{l}=\left(  Z_{i_{1}}\right)
^{l},\\j_{1}(l)   =\min\left\{  j:\mathfrak{c}_{j}\text{ }\not \subset \text{
}\mathfrak{h}_{1}(l)\right\}  .
\end{align*}
By induction, for any $k\in\{1,2,\cdots,n\},$ we define%
\begin{align}
\label{polar}i_{k}(l)  &  =\min\{j:\mathfrak{c}_{j}\cap \mathfrak{h}_{k-1}%
(l)\not \subset \mathfrak{h}_{k-1}(l)^{l}\},\\
\mathfrak{h}_{k}(l)  &  =\left(  \mathfrak{h}_{k-1}(l)\cap\mathfrak{c}_{i_{k}%
}\right)  ^{l}\cap\mathfrak{h}_{k-1}(l),\\
j_{k}(l)  &  =\min\left\{  j:\mathfrak{c}_{j}\cap \mathfrak{h}_{k-1}(l)\text{
}\not \subset \text{ }\mathfrak{h}_{k}(l)\right\}  .
\end{align}
Finally, put $\mathbf{e}(l)=\mathbf{i}(l)\cup\mathbf{j}(l),$ where $\mathbf{i}(l)=\{i_{k}(l):1\leq k\leq d\},$ and $\mathbf{j}(l)=\{j_{k}(l):1\leq k\leq d\}$. An interesting
well-known fact is that $\mathrm{card}\left(  \mathbf{e}(l)\right)  $ is
always even. Also, observe the sequence $\{i_{k}:1\leq k\leq d\}$ is an
increasing sequence and, $i_{k} <j_{k}$ for $1\leq k\leq d$.

Following Definition $2$ \cite{Didier} , let $\mathcal{P}$ be a partition of the linear dual of the Lie algebra $\mathfrak{g}$.
\begin{definition} We say $\mathcal{P}$ is an \textbf{orbital
stratification} of $\mathfrak{g}^{\ast}$ if the following conditions are satisfied
\end{definition}

\begin{enumerate}
\item Each element $\Omega$ in $\mathcal{P}\ $is $G$-invariant.

\item For each $\Omega$ in $\mathcal{P}$, the coadjoint orbits in
$\Omega$ have the same dimension.

\item There is a linear ordering on $\mathcal{P}$ such that for each
$\Omega\in\mathcal{P},$ $${\displaystyle\bigcup}
\left\{  \Omega^{\prime}|\Omega^{\prime}\leq\Omega\right\}  $$ is a Zariski open subset of $\mathfrak{g}^{\ast}.$
\end{enumerate}

The elements $\Omega$ belonging to a stratification are called
\textbf{layers} of the dual space $\mathfrak{g}^{\ast}.$

\begin{definition}
Given any subset of $\mathbf{e}$ of $\left\{  1,2,\cdots,n\right\}  ,$ we
define the set
\[
\Omega_{\mathbf{e}}=\left\{  l\in\mathfrak{g}^{\ast}|\mathbf{e}(l)=\mathbf{e}%
\right\}
\]
which is $G$-invariant$.$ The collection of non-empty $\Omega_{\mathbf{e}}$
forms a partition of $\mathfrak{g}^{\ast}.$ Such partition is called a
\textbf{coarse stratification }of $\mathfrak{g}^{\ast}$. Given $\mathbf{e}%
(l)=\left\{  i_{1},\cdots,i_{d}\right\}  \cup\left\{  j_{1},\cdots
,j_{d}\right\}  ,$ we define $$\Omega_{\mathbf{e},\mathbf{j}}=\left\{
l\in\mathfrak{g}^{\ast}| \:\mathbf{e}(l)=\mathbf{e}\text{ and }\mathbf{j}%
(l)=\mathbf{j}\right\}.$$ The collection of non-empty $\Omega_{\mathbf{e,j}}$
forms a partition of $\mathfrak{g}^{\ast}$ called the \textbf{fine
stratification} of $\mathfrak{g}^{\ast},$ and the elements $\Omega
_{\mathbf{e},\mathbf{j}}$ are called \textbf{fine layers}.
\end{definition}

We keep the notations used in \cite{Didier}.

\begin{enumerate}
\item We fix an adaptable basis, an open dense layer $\Omega_{\mathbf{e}%
,\mathbf{j}}.$ We let $\mathfrak{c}_{0}=\{0\},$ and we define the following
sets:
\begin{align}
I  &  =\left\{  0\leq j\leq n+r\text{ }:\mathfrak{c}_{j}=\overline
{\mathfrak{c}_{j}}\right\}  ,\label{data}\\
j^{\prime}  &  =\max\left(  \left\{  0,1,\cdots,j-1\right\}  \cap I\right)
,\nonumber\\
j^{\prime\prime}  &  =\min\left(  \left\{  j,j+1,\cdots,n+r\right\}  \cap
I\right)  ,\nonumber\\
K_{0}  &  =\left\{  1\leq k\leq d\text{ }:\mathfrak{\ }i_{k}^{\prime\prime
}-i_{k}^{\prime}=1\right\}  ,\nonumber\\
K_{1}  &  =\left\{  1\leq k\leq d\text{ }:\mathfrak{\ }i_{k}\notin I\text{ and
}i_{k}+1\notin\mathbf{e}\right\}  ,\nonumber\\
K_{2}  &  =\left\{  1\leq k\leq d\text{ }:\mathfrak{\ }i_{k}-1\in
\mathbf{j}\backslash I\right\}  ,\nonumber\\
K_{3}  &  =\left\{  1\leq k\leq d\text{ }:\mathfrak{\ }i_{k}\notin I\text{ and
}i_{k}+1\in\mathbf{j}\right\}  ,\nonumber\\
K_{4}  &  =\left\{  1\leq k\leq d\text{ }:\mathfrak{\ }i_{k}\notin I\text{ and
}i_{k}+1\in\mathbf{i}\right\}  ,\nonumber\\
K_{5}  &  =\left\{  1\leq k\leq d\text{ }:\mathfrak{\ }\text{ }i_{k}%
-1\in\mathbf{i\backslash}I\right\}  .\nonumber
\end{align}
 We remark here that $$\mathbf{i}=\bigcup_{j=0}^{5}\{i_k: k \in K_j\}.$$

\item We gather some data corresponding to the fixed fine layer $\Omega
_{\mathbf{e},\mathbf{j}}$. For each $j\in\mathbf{e,}$ we
define recursively the rational function $Z_{j}:\Omega\rightarrow
\mathfrak{c}_{j^{\prime\prime}}$ such that for $k\in\left\{  1,2,\cdots
,d\right\}  ,$
\begin{align}
V_{1}(l)  &  =Z_{i_{1}}(l),U_{1}(l)=Z_{j_{1}}(l),\label{data 2}\\
V_{k}(l)  &  =\rho_{k-1}\left(  Z_{i_{k}}(l),l\right)  ,U_{k}(l)=\rho
_{k-1}\left(  Z_{j_{k}}(l),l\right)  ,\nonumber\\
Z_{i_{k}}(l)  &  =\beta_{1,k}(l)\operatorname{Re}Z_{i_{k}}+\beta
_{2,k}(l)\operatorname{Im}Z_{i_{k}},\nonumber\\
Z_{j_{k}}(l)  &  =\alpha_{1,k}(l)\operatorname{Re}Z_{j_{k}}+\alpha
_{2,k}(l)\operatorname{Im}Z_{j_{k}},\nonumber\\
\alpha_{1,k}  &  =l\left[  \operatorname{Re}Z_{j_{k}},V_{k}(l)\right]
,\alpha_{2,k}=l\left[  \operatorname{Im}Z_{j_{k}},V_{k}(l)\right]  .\nonumber
\end{align}%
\[
\rho_{k}\left(  Z,l\right)  =\rho_{k-1}\left(  Z,l\right)  -\frac{l\left[
\rho_{k-1}\left(  Z,l\right)  ,U_{k}(l)\right]  }{l\left[  V_{k}%
(l),U_{k}(l)\right]  }V_{k}(l)-\frac{l\left[  \rho_{k-1}\left(  Z,l\right)
,V_{k}(l)\right]  }{l\left[  U_{k}(l),V_{k}(l)\right]  }U_{k}(l)
\]
and $\rho_{0}(\cdot, l)$ is the identity map.

\begin{description}
\item[(a)] If $k\in K_{0},\beta_{1,k}(l)=1,$ and $\beta_{2,k}(l)=0.$

\item[(b)] If $k\in K_{1},\beta_{1,k}(l)=l\left(  \left[  \rho_{k-1}\left(
Z_{j_{k}},l\right)  ,\operatorname{Re}Z_{i_{k}}\right]  \right)  ,$ and
\[
\beta_{2,k}(l)=l\left(  \left[  \rho_{k-1}\left(  Z_{j_{k}},l\right)
,\operatorname{Im}Z_{i_{k}}\right]  \right)  .
\]
\item[(c)] If $k\in K_{2},i_{k}-1=j_{k},$ $\beta_{1,k}(l)=-\alpha_{2},_{k}(l)$
and $\beta_{2,k}(l)=-\alpha_{1},_{k}(l).$

\item[(d)] If $k\in K_{3},\beta_{1,k}(l)=0,\beta_{2,k}(l)=1.$

\item[(e)] If $k\in K_{4}(K_{5}$ is covered here too) and if $Z_{j_{k+1}%
}=\overline{Z_{j_{k}}},$ then $\beta_{1,k}(l)=1,\beta_{2,k}(l)=0,$ and
\[
Z_{i_{k}+1}(l)=-\left(  \frac{l\left[  U_{k}(l),\operatorname{Im}Z_{i_{k}%
}\right]  }{l\left[  U_{k}(l),\operatorname{Re}Z_{i_{k}}\right]  }\right)
\operatorname{Re}Z_{i_{k}+1}-\operatorname{Im}Z_{i_{k}+1}.
\]

\end{description}

\item Let $C_{j}=\ker\gamma_{j}\cap\mathfrak{g,a}_{j}(l)=\left(
\mathfrak{g}_{j^{\prime}}^{l}\cap C_{j}\right)  /\left(  \mathfrak{g}%
_{j^{\prime\prime}}^{l}\cap C_{j}\right)  ,$ we define the set $\varphi
(l)\subset\mathbf{i}$ such that $\varphi(l)=\left\{  j\in\mathbf{e|}%
\mathfrak{a}_{j}(l)= \{0\}\right\}  ,$ and
\[
\mathbf{b}_{j}(l)=\frac{\gamma_{j}\left(  U_{k}(l)\right)  }{l\left[
Z_{j},U_{k}(l)\right]  }.
\]
The collection of sets $\Omega_{\mathbf{e,j},\varphi}=\left\{  l\in
\Omega_{\textbf{e},\textbf{j}}|\varphi(l)=\varphi\right\}  $ forms a partition of $\mathfrak{g}%
^{\ast}$, refining the fine stratification which, we call the
\textbf{ultrafine stratification} of $\mathfrak{g}^{\ast}.$

\item Letting $\Omega_{\mathbf{e,j},\varphi}$ be a layer obtained by refining
the fixed fine layer $\Omega_{\mathbf{e},\mathbf{j}},$ and gathering the data
\[
Z_{j}(l),\mathbf{e},\varphi(l),\mathbf{b}_{j}(l),
\]
the cross-section for the coadjoint orbits of $\Omega$ is given by the set
\begin{equation}
\Sigma=\left\{  l\in\Omega:l\left(  Z_{j}(l)\right)  =0,j\in\mathbf{e}%
\backslash\varphi\text{ and }\left\vert \mathbf{b}_{j}(l)\right\vert
=1\ ,j\in\varphi\text{ }\right\}  . \label{crs}%
\end{equation}
\end{enumerate}
Let us now offer some concrete examples. 
\begin{example} Let $\mathfrak{g}$ be a Lie algebra spanned by $\{Z,Y,X,A\}$ with the following non-trivial Lie brackets: 
$$ [X,Y]=Z, [A,X+iY]=(1+i)(X+iY), [A,Z]=2Z.$$ An adaptable basis is $\{Z,X+iY,X-iY,A \}$ and an arbitrary linear functional is written as $l=(z,x+iy,x-iy,a)$. Here $I=\{0,1,3,4\}, 1'=0,2'=1,3'=1,4'=3, 4''=1,2''=3, 3''=3,\text{and } 4''=4.$ Put $\mathbf{e}=\{1,2,3,4\},$ and $\mathbf{j}=\{3,4\}.$ Next, it is easy to see that $1\in K_0$ and $2\in K_3.$ Moreover, we have  $$Z_{i_1}(l)=V_1(l)=Z, Z_{j_1}(l)=U_1(l)=A,Z_{i_2}(l)=Y,V_2(l)=\rho_1(Y,l)=Y-\dfrac{x+y}{2z}Z, $$ and $$ Z_{j_2}(l)=X,U_2(l)=\rho_1(X,l)=X-\dfrac{x-y}{2z}Z. $$ Then $\varphi=\{1\}$ and  $\Omega_{\mathbf{e,j}}=\{(z,x+iy,x-iy,a) : z \neq 0\}$ and $$\Sigma =\{(z,x+iy,x-iy,a)\in \Omega : |z|=1,a=x=y=0\}.$$ 
\end{example}
\begin{example} Let $\mathfrak{g}$ be a Lie algebra spanned by $$\{Z_1,Z_2,Y,X_1,X_2,A\}$$ with the following non-trivial Lie brackets: 
$$ [X_j,Y]=Z_j, [A,X_1+iX_2]=(1+i)(X_1+iX_2), [A,Z_1+iZ_2]=(1+i)(Z_1+iZ_2).$$ We choose an adaptable basis 
$$\{Z_1+iZ_2,Z_1-iZ_2,Y,X_1+iX_2,X_1-iX_2,A \}$$ for $\mathfrak{c}.$ We compute here that $I=\{0,2,3,5,6\},$ and
$1'=0,2'=0,3'=2,4'=3,5'=3,6'=5, 1''=2,2''=2,3''=3,4''=5,5''=5,6''=6. $
Pick $\mathbf{e}=\{1,3,4,6\},$ and  $\mathbf{j}=\{6,4\}$. In this example, the set $K_1$ contains $1$, $K_0$ contains $2$. Next, with some simple computations, we obtain $$Z_{i_1}(l)=(z_1-z_2)Z_1+(z_1+z_2)Z_2, Z_{j_1}=A, Z_{i_2}(l)=Y,Z_{j_2}=z_1 X_1+z_2 X_2.$$

Clearly $\varphi=\{1\},$ the corresponding layer is $\Omega_{\mathbf{e},\mathbf{j}}=\{(z,\overline{z},y,x,\overline{x}) : z \neq 0 \}$ and the corresponding cross-section is $$\Sigma=\{(z,\overline{z},y,x,\overline{x}) : |z|=1, a=y=0, \mathrm{Re}(\overline{z}x)=0 \} . $$\end{example}

Now, that we are introduced to the general construction, we will focus our attention to $N$ which is the
Lie group of the nilradical of $\mathfrak{g}$. $N$ being an exponential
solvable Lie group also, Formula \ref{crs} is valid. Let us recall the following well-known
facts. The first one is due to Kirillov, and the second one is an application
of the `Mackey Machine' (see \cite{Mackey}).

\begin{lemma}
Let $f\in\mathfrak{n}^{\ast}$, and $\widehat{N}$ the set of unitary
irreducible representations of $N$ up to equivalence. Let $\mathfrak{n}^{\ast
}/N=\{N\cdot f:f\in\mathfrak{n}^{\ast}\}$ be the set of coadjoint orbits.
There exists a unique bijection between $\mathfrak{n}^{\ast}/N$ and
$\widehat{N}$ via Kirillov map. Thus, the construction of a measurable
cross-section for the coadjoint orbits is a natural way to parametrize
$\widehat{N}$.
\end{lemma}

\begin{lemma}
The set of unitary irreducible representations of $G$, $\widehat{G}$ is a
fiber set with $\widehat{N}/H$ as base, and fibers $\widehat{H_{\lambda}}$,
where $H_{\lambda}$ is a closed subgroup of $H$ stabilizing the coadjoint
action of $H$ on the linear functional $\lambda$ .
\end{lemma}

We aim here to construct an $H$-invariant cross-section for the the coadjoint
orbits of $N$ in $\mathfrak{n}^{\ast}$. We consider the nilradical
$\mathfrak{n}$ of $\mathfrak{g}$ instead of $\mathfrak{g}$, and we go through
the procedure described earlier. We first obtain an adaptable basis $\left\{
Z_{1},\cdots,Z_{n}\right\}  $ for the complexification of the Lie algebra
$\mathfrak{n}$ which we denote by $\mathfrak{m}$. Notice that, $\left\{
Z_{1},\cdots,Z_{n},A_{1},\cdots,A_{\dim(\mathfrak{h})}\right\}  $ is then an
adaptable basis for $\mathfrak{g}$. First, fixing a dense open layer $\Omega\subset\mathfrak{g}^{\ast}$ and
$f\in\Omega$, we obtain the jump indices corresponding to the generic layer of
$\mathfrak{g}^{\ast}$.
\begin{align*}
\mathbf{i}^{\circ}(f)  &  =\left\{  i_{1},\cdots,i_{d^{\circ}}\right\} \\
\mathbf{j}^{\circ}(f)  &  =\left\{  j_{1},\cdots,j_{d^{\circ}}\right\} \\
\mathbf{e}^{\circ}(f)  &  =\left\{  i_{1},\cdots,i_{d^{\circ}}\right\}
\cup\left\{  j_{1},\cdots,j_{d^{\circ}}\right\}  .
\end{align*}

Second, let $\Omega_{\mathbf{e}^{\circ}\mathbf{j}^{\circ}}$ be a fixed fine
layer obtained by refining $\Omega$. Given any subset $\mathbf{e}^{\circ
}\mathbf{\subseteq\{}1,\cdots,n\}$, the non-empty sets $\Omega_{\mathbf{e}
^{\circ}\mathbf{j}^{\circ}}$ are characterized by the Pfaffian of the
skew-symmetric matrix $M_{\mathbf{e}^{\circ}}(f)=\left[  f\left[  Z_{i}%
,Z_{j}\right]  \right]  _{i,j\in\mathbf{e}^{\circ}}.$ Referring to the
procedure described in (\ref{data}) and (\ref{data 2}), we obtain
\[
Z_{_{i_{1}^{\circ}}}\left(  f\right)  ,Z_{j_{1}^{\circ}}\left(  f\right)
\cdots Z_{i_{d^{\circ}}}\left(  f\right)  ,Z_{j_{d^{\circ}}}\left(  f\right)
,
\]
and we have the polarizing sequence $\mathfrak{m=}$ $\mathfrak{h}%
_{0}(l)\supseteq\mathfrak{h}_{1}(l)\supseteq\cdots\supseteq\mathfrak{h}%
_{d^{\circ}}(l)$. Thirdly, we compute the following data:
\begin{align*}
&  I,j^{\prime},j^{\prime\prime},K_{0},K_{1},K_{2},K_{3},K_{4},K_{5}
,V_{1}(f),\\
&  \cdots,V_{d^{\circ}}(f),U_{1}(f),\cdots,U_{d^{\circ}}(f),\varphi
(f),b_{j}(f)
\end{align*}
corresponding to our fine layer $\Omega_{\mathbf{e}^{\circ}\mathbf{j}^{\circ}%
}$ as described in (\ref{data}) and (\ref{data 2}). Finally, gathering all the data, we first notice that $\varphi(f)=\emptyset$, since according to Proposition 4.1 in \cite{Didier}, $\mathfrak{a}_{j}\left(
l\right)  =0$ if and only if $\gamma_{j}\left(  U_{k}\left(  l\right)
\right)  \neq0$ for $j=i_{k}.$ As shown in \cite{Didier}, an $H$-invariant cross-section for the
coadjoint $N$ orbits for $\Omega_{\mathbf{e}^{\circ}}$ is given by
\begin{equation}
\Lambda=\left\{  f\in\Omega_{\mathbf{e}^{\circ},\mathbf{j}^{\circ}}\text{
}:f\left(  Z_{j}(f)\right)  =0,j\in\mathbf{e}^{\circ}\right\}  .
\label{Lambda}%
\end{equation}
Following the proof of Theorem 4.2 in \cite{Didier}, we have three separate cases
\begin{description}
\item[Case 1] If $j\in I$ or if $j\not \in I$ and $j+1\in\mathbf{e}^{\circ}$
then $f\left(  Z_{j}(f)\right)  =0$ is equivalent to $f\left(  Z_{j}\right)
=0$. 
\item[Case 2] If $j\not \in I$, $j+1\not \in \mathbf{e}^{\circ},$ and
$j=i_{k}$ then
\begin{align*}
f(Z_{j}(f))  &  =f([\rho_{k-1}(Z_{j_{k}},f),\operatorname{Re}Z_{j}%
])\operatorname{Re}f(Z_{j})\\
&  +f[\rho_{k-1}(Z_{j_{k}},f),\operatorname{Im}Z_{j}]\operatorname{Im}%
f(Z_{j}).
\end{align*}
\item[Case 3] If $j\not \in I$, $j+1\not \in \mathbf{e}^{\circ},$ and
$j=j_{k}$ then the equation $f\left(  Z_{j}(f)\right)  =0$ is equivalent to
\[
\operatorname{Re}(f[\rho_{k-1}(\overline{Z_{j}},f), \operatorname{Re}
Z_{i_{k}} ] f(Z_{j})=\operatorname{Re}(f[\rho_{k-1}(\overline{Z_{j}},f),
\operatorname{Im} Z_{i_{k}} ] f(Z_{j})=0.
\]
\end{description}
\begin{remark}
\label{remark} If the assumptions of Case 1 hold for all elements of
$\mathbf{e}^{\circ}$ then $$
\Lambda=\left\{  f\in\Omega_{\mathbf{e}^{\circ},\mathbf{j}^{\circ}}\text{
}:f\left(  Z_{j}\right)  =0,j\in\mathbf{e}^{\circ}\right\}.$$
\end{remark}
\begin{example} Let $\mathfrak{g}$ be a nilpotent Lie algebra spanned by $\{Z_1,Z_2,Y_1,Y_2,X_1,X_2\}$ with the following non-trivial Lie brackets: $[X_j,Y_j]=Z_j.$ Choosing the following adaptable basis $$\{Z_1+iZ_2,Z_1-iZ_2,Y_1+iY_2,Y_1-iY_2,X_1+iX_2,X_1-iX_2\},$$ letting $\mathbf{e}^{\circ}=\{3,4,5,6\}$, and $\mathbf{j}^{\circ}=\{5,6\}$ then $$\Omega_{\mathbf{e}^{\circ},\mathbf{j}^{\circ}}=\{(z,\overline{z},y,\overline{y},x,\overline{x}) : z\neq 0\}$$ and $$\Lambda=\{(z,\overline{z},y,\overline{y},x,\overline{x})\in \Omega_{\mathbf{e}^{\circ},\mathbf{j}^{\circ}} : x=y=0 \}. $$ \end{example}
 Now, we will compute a general formula a smooth cross-section for the $G$-orbits in
some open dense set in $\mathfrak{g}^{\ast}$. Let $\lambda:\Omega
_{\mathbf{e}^{\circ}\mathbf{j}^{\circ}}\rightarrow\Lambda$ be the
cross-section mapping, for each $f\in\mathfrak{n}^{\ast}$, we define
$\nu\left(  f\right)  =\left\{  1\leq j\leq n:f\left(  Z_{j}\right)
\neq0\right\}  .$ Put
\[
\mathfrak{h}\left(  f\right)  = {\displaystyle\bigcap\limits_{j\in\nu\left(
f\right)  }} \ker\gamma_{j},
\]
and let $\Lambda_{\nu}=\left\{  f\in\Lambda:\nu\left(  f\right)  =\nu\right\}
. $ Observe that $\mathfrak{h}(f)$ is the Lie algebra of the stabilizer
subgroup (a subgroup of $H$) of the linear functional $f$. For any
$f\in\Lambda_{\nu}$, since we have a diagonal action, then $\mathfrak{h}(f)$
is independent of $f$ and is equal to some constant subalgebra $\mathfrak{k}%
\subset\mathfrak{h}$. 

\begin{lemma}
There exists $\nu\subseteq\left\{  1,\cdots,n\right\}  $ such that
$\Lambda_{\nu}$ is dense and Zariski open in $\Lambda$, and letting $\pi$ be
the projection or restriction mapping from $\mathfrak{g}^{\ast}$ onto
$\mathfrak{n}^{\ast}$, and $\Omega_{\nu}=\pi^{-1}\circ\lambda^{-1}\left(
\Lambda_{\nu}\right)  , $ then $\Omega_{\nu}$ is Zariski open in
$\mathfrak{g}^{\ast}$.
\end{lemma}

\begin{proof}
It suffices to let $\nu=\left\{  1,\cdots,n\right\}  \backslash\mathbf{e}
^{\circ}.$ Notice that
\[
\Lambda_{\nu}=\left\{  f\in\Lambda\text{ }:\nu\left(  f\right)  =\left\{
1,\cdots,n\right\}  \backslash\mathbf{e}^{\circ}\text{ }\right\}
\]
is dense and Zariski open in $\Lambda$. Additionally, we observe that for
$f\in\Lambda_{\nu},$ and $j\in$ $\left\{  1,\cdots,n\right\}  \backslash
\mathbf{e}^{\circ},f\left(  Z_{j}\right)  \neq0.$ Next, $\Omega_{\nu}$ is
Zariski open in $\mathfrak{g}^{\ast}$ since the projection map is continuous,
and the cross-section mapping is rational and smooth (see \cite{Didier}).
\end{proof}

\begin{lemma}
If $l\in\Omega_{\nu},\mathbf{e}\left(  l\right)  $ is the set of jump indices
for $\Omega_{\nu}$ such that
\begin{align*}
\mathbf{e}(l)  &  =\left\{  i_{1},\cdots,i_{d}\right\}  \cup\left\{
j_{1},\cdots,j_{d}\right\}  ,\\
\mathbf{i}(l)  &  =\left\{  i_{1},\cdots,i_{d}\right\}  ,\\
\mathbf{j}(l)  &  =\left\{  j_{1},\cdots,j_{d}\right\}
\end{align*}
then $\max\mathbf{i}(l)$ $\leq\dim\mathfrak{n}$.
\end{lemma}

\begin{proof}
Let us assume by contradiction that there exists some jump index $i_{t}%
\in\mathbf{i}(l)$ such that $Z_{i_{t}}\in\mathfrak{h}$. Because, jump indices
always come in pairs, and because $j_{t}>i_{t}$, then $Z_{j_{t}}%
\in\mathfrak{h}$. However, since $\mathfrak{h}$ is commutative, then
$l[Z_{i_{t}},Z_{j_{t}}]=0$. This is a contradiction.
\end{proof}

\begin{lemma}
For any $l\in\Omega_{\nu}$, and for all $j\in$ $\left(  \mathbf{e}\left(
l\right)  \backslash\mathbf{e}^{\circ}\right)  \backslash\mathbf{i}(l),$
$Z_{j}\in\mathfrak{h}.$
\end{lemma}

\begin{proof}
We have $\mathbf{e}\left(  l\right)  =\mathbf{e}^{\circ}\overset{\cdot}{\cup
}\left\{  i_{s_{1}},\cdots,i_{s_{r}}\right\}  \overset{\cdot}{\cup}\left\{
j_{s_{1}},\cdots,j_{s_{r}}\right\}  $. If $j\in$ $\left(  \mathbf{e}\left(
l\right)  \backslash\mathbf{e}^{\circ}\right)  \backslash\mathbf{i}(l)$, then
$j\in\mathbf{j}(l)\backslash\mathbf{e}^{\circ},$ and there exists some $k$
such that $Z_{j}=Z_{j_{s_{k}}}.$ Assume that $Z_{j_{s_{k}}}\in\mathfrak{n}$.
Since $j_{s_{k}}\notin\mathbf{e}^{\circ}$, there must exist some jump index
$i_{s_{k}}$ such that $i_{s_{k}}<j_{s_{k}}$ and $l[Z_{i_{s_{k}}},Z_{j_{s_{k}}%
}]\neq0$. Since $Z_{i_{s_{k}}}$ also belongs to $\mathfrak{n}$, then letting
$\pi(l)=f$, $f[Z_{i_{s_{k}}},Z_{j_{s_{k}}}]\neq0$. Thus, both $i_{s_{k}%
},j_{s_{k}}\in\mathbf{e}^{\circ}$ which is a contradiction according to our assumption.
\end{proof}

We observe that the choice of an adaptable basis mainly relies on the choice for
an adaptable basis for the nilpotent Lie algebra. Any permutation of the basis
elements of $\mathfrak{h}$ will not affect the `adaptability' of the basis.
Without loss of generality, we will assume that we have the following
adaptable basis for $\mathfrak{g}:$ $$ \left\{  Z_{1},\cdots Z_{n},A_{m}%
,\cdots,A_{r+1},A_{r},\cdots,A_{2},A_{1}\right\}$$ such that $A_{r}%
=Z_{j_{s_{r}}},\cdots,A_{1}=Z_{j_{s_{1}}}.$ Additionally, we assume that the
basis elements $A_{r}\cdots A_{2},A_{1}$ with weight $\left\{  \gamma
_{r},\cdots,\gamma_{1}\right\}  $ are chosen such that $\operatorname{Re}%
\left(  \gamma_{t}\left(  A_{t}\right)  \right)  =1,\text{ }\gamma_{t}\left(
A_{t^{\prime}}\right)  =0,t\neq t^{\prime}. $

\begin{lemma}
For any $l\in\Omega_{\nu},\varphi\left(  l\right)  =\left\{  i_{s_{1}}%
,\cdots,i_{s_{r}}\right\}  $.
\end{lemma}

\begin{proof}
We already have that $\varphi\left(  l\right)  \subseteq\left\{  i_{s_{1}%
},\cdots,i_{s_{r}}\right\}  $. We only need to show that for any $j=i_{s_{1}%
},j\in\varphi\left(  l\right)  $. By definition, $\varphi\left(  l\right)
=\left\{  j\in\mathbf{e} : \mathfrak{a}_{j}\left(  l\right)  =0\right\}  $ and
according to Proposition 4.1 in \cite{Didier}, $\mathfrak{a}_{j}\left(
l\right)  =0$ if and only if $\gamma_{j}\left(  U_{k}\left(  l\right)
\right)  \neq0$ for $j=i_{k}.$ In order to prove the proposition, it suffices
to show that $\gamma_{i_{s_{k}}}\left(  U_{k}\left(  l\right)  \right)  =0$.
\begin{align*}
U_{k}\left(  l\right)   &  =\rho_{k-1}\left(  Z_{j_{s_{k}}}\left(  l\right)
,l\right) \\
&  =\rho_{k-1}\left(  A_{s_{k}}\right)  =\rho_{k-1}\left(  A_{k}\right) \\
&  =\rho_{k-2}\left(  A_{k},l\right)  -\frac{l\left[  \rho_{k-2}\left(
A_{k},l\right)  ,U_{k-1}(l)\right]  }{l\left[  V_{k-1}(l),U_{k-1}(l)\right]
}V_{k-1}(l)-\\
&  \frac{l\left[  \rho_{k-2}\left(  A_{k},l\right)  ,V_{k-1}(l)\right]
}{l\left[  U_{k-1}(l),V_{k-1}(l)\right]  }U_{k-1}(l).
\end{align*}
A straightforward computation shows that for some coefficients $c_{t}$%
\begin{align*}
\gamma_{i_{s_{k}}}\left(  U_{k}\left(  l\right)  \right)   &  =\gamma
_{k}\left(  A_{k}\right)  -c_{k-1}\gamma_{k}\left(  A_{k-1}\right)
-\cdots-c_{1}\gamma_{1}\left(  A_{1}\right) \\
&  =\gamma_{k}\left(  A_{k}\right)  \neq0.
\end{align*}
This completes the proof.
\end{proof}

\begin{proposition}
\label{cross} Let $\mathfrak{g}=\mathfrak{n}\times\mathfrak{k}\times
\mathfrak{a}$ where, $\mathfrak{h}=\mathfrak{k}\times\mathfrak{a}$. The
cross-section for the $G$-orbits in $\Omega_{\nu}$ is
\[
\Sigma=\left\{  l\in\Omega_{\nu}:l=\left(  f,k,0\right)  ,f\in\Sigma^{\circ
},k\in\mathfrak{k}^{\ast}\right\}  .
\]
Letting $\pi:\mathfrak{g}^{\ast}\rightarrow\mathfrak{n}^{\ast}$ be the
projection map,
\[
\pi\left(  \Sigma\right)  =\Sigma^{\circ}=\left\{  l\in\Lambda_{\nu
}:\left\vert l\left(  Z_{j}\right)  \right\vert =1\text{ }\forall j\in\left\{
i_{s_{1}},\cdots,i_{s_{r}}\right\}  \right\}  .
\]

\end{proposition}

\begin{proof}
Let $\pi(l)=f$. So far, we have shown that $\mathbf{e}\left(  l\right)
=\mathbf{e}^{\circ}\overset{\cdot}{\cup}\varphi\left(  l\right)
\overset{\cdot}{\cup}\left\{  j_{s_{1}},\cdots,j_{s_{r}}\right\}  $. Using the
description of the cross-section described in \cite{Didier},
\[
\Sigma=\left\{  l\in\Omega_{\nu}:l\left(  Z_{j}\left(  l\right)  \right)
=0\text{ for }j\in\mathbf{e\backslash}\varphi,\text{ and }\left\vert
\mathbf{b}_{j}\left(  l\right)  \right\vert =1\text{ for }j\in\varphi\right\}
.
\]
For $l\in\mathfrak{g}^{\ast}$, if $j\in\mathbf{e\backslash}\varphi$ then
$j\in\mathbf{e}^{\circ}\overset{\cdot}{\cup}\left\{  j_{s_{1}},\cdots
,j_{s_{r}}\right\}  .$ For $j\in\mathbf{e}^{\circ},l\left(  Z_{j}\left(
l\right)  \right)  =f\left(  Z_{j}\left(  f\right)  \right)  =0$ and for
$j\in\left\{  j_{s_{1}},\cdots,j_{s_{r}}\right\}  ,l\left(  Z_{j}\left(
l\right)  \right)  =0.$ Thus, $A_{j}=0$ for $j\in\left\{  j_{s_{1}}%
,\cdots,j_{s_{r}}\right\}  .$ Next, for $j\in\varphi(l)=\left\{  i_{s_{1}%
},\cdots,i_{s_{r}}\right\}  ,$%
\[
\left\vert \mathbf{b}_{j}\left(  l\right)  \right\vert =\left\vert
\frac{\gamma_{j}\left(  U_{k}\left(  l\right)  \right)  }{l\left[  Z_{j}%
,U_{k}\left(  l\right)  \right]  }\right\vert =\left\vert \frac{\gamma
_{j}\left(  A_{k}\right)  }{l\left[  Z_{j},A_{K}\right]  }\right\vert
=\left\vert \frac{1}{l\left(  Z_{j}\right)  }\right\vert =1\Rightarrow
\left\vert l\left(  Z_{j}\right)  \right\vert =1.
\]
Thus, we conclude that $\Sigma=\left\{  l\in\Omega_{\nu}:l=\left(
f,k,0\right)  ,f\in\Sigma^{\circ},k\in\mathfrak{k}^{\ast}\right\}  $ where
\[
\Sigma^{\circ}=\left\{  l\in\Lambda_{\nu}:\left\vert l\left(  Z_{j}\right)
\right\vert =1\text{ }, j\in\left\{  i_{s_{1}},\cdots,i_{s_{r}}\right\}
\right\}  .
\]

\end{proof}

Throughout the remainder of this paper, we will also use the symbol $\simeq$
to denote a homeomorphism between two topological spaces.

\begin{proposition}
\label{proposition18} $\Sigma^{\circ}$is a cross-section for the $H$-orbits in
$\Lambda_{\nu}.$ In other words,
\[
\Sigma^{\circ}=\pi\left(  \Sigma\right)  \simeq\Lambda_{\nu}/H.
\]

\end{proposition}

\begin{proof}
The set $\Lambda_{\nu}$ is an $H$ invariant cross-section for the $N$
coadjoint orbits of a fixed layer $\Omega_{\mathbf{e}^{\circ}\mathbf{j}%
^{\circ}}$, while the set $\Sigma$ is a cross-section for the $G$ coadjoint
orbits of for $\Omega_{\nu}.$ In order to prove the proposition, we must show
that each $H$-orbit of any arbitrary element inside $\Lambda_{\nu}$ meets the
set $\Sigma^{\circ}$ at exactly one unique point, and also any arbitrary point
in $\Sigma^{\circ}$ belongs to an $H$ orbit of some linear functional
belonging to $\Lambda_{\nu}.$ We start by showing that $H\cdot f\cap
\Sigma^{\circ}$ is a non empty set for $f\in\Lambda_{\nu}$. Given $f\in
\Lambda_{\nu},$ we consider the element $\left(  f,k,0\right)  \in\Omega_{\nu
}$ such that $f=\pi\left(  \left(  f,k,0\right)  \right)  .$ We know there
exits an element $x\in\Sigma$ such that $g\cdot\left(  f,k,0\right)  =x$, for
some $g\in G$. In fact, let $g=\left(  n,1\right)  \left(  1,h\right)  .$ If
$\left(  n,1\right)  \left(  1,h\right)  \cdot\left(  f,k,0\right)  =x,$ then
$\pi\left(  \left(  n,1\right)  \left(  1,h\right)  \cdot\left(  f,k,0\right)
\right)  =\pi\left(  x\right)  $, and $\left(  n,1\right)  \pi\left(  \left(
1,h\right)  \cdot\left(  f,k,0\right)  \right)  =\pi\left(  x\right)
\in\Lambda_{\nu}.$ Thus, $\left(  n,1\right)  $ stabilizes $\pi\left(  \left(
1,h\right)  \cdot\left(  f,k,0\right)  \right)  $ implying that $\pi\left(
\left(  1,h\right)  \cdot\left(  f,k,0\right)  \right)  =\pi\left(  x\right)
\in\Lambda_{\nu}.$ Since $$\pi\left(  \left(  1,h\right)  \cdot\left(
f,k,0\right)  \right)  =\pi\left(  \left(  h\cdot f,k,0\right)  \right)
=h\cdot f$$ $h\cdot f\in\pi\left(  \Sigma\right)  =\Sigma^{\circ}.$ Next, let
us assume that there exits $h,$ and $h^{\prime}\in H$ such that $f\in$
$\Lambda_{\nu}$ and $h\cdot f,h^{\prime}\cdot f\in\Sigma^{\circ}$ with $h\cdot
f\neq h^{\prime}\cdot f.$ Now consider $\left(  h^{\prime}\cdot f,k,0\right)
,\left(  h\cdot f,k,0\right)  \in\Sigma.$ We have,%
\begin{align*}
\left(  h\cdot f,k,0\right)   &  =(1,h)\cdot\left(  f,k,0\right) \\
\left(  h^{\prime}\cdot f,k,0\right)   &  =(1,h^{\prime})\cdot\left(
f,k,0\right)  .
\end{align*}
Both $\left(  h\cdot f,k,0\right)  ,\left(  h^{\prime}\cdot f,k,0\right)$ are elements of the $G$-orbit of $(f,k,0)$, and since the elements $\left(  h\cdot
f,k,0\right)  $, and $\left(  h^{\prime}\cdot f,k,0\right)  $ also belong to
the cross-section $\Sigma$ then $\left(  h\cdot f,k,0\right)  =\left(
h^{\prime}\cdot f,k,0\right)  .$ The latter implies that $h\cdot f=h^{\prime
}\cdot f$. We reach a contradiction. We conclude that $\pi\left(
\Sigma^{\circ}\right)  =\pi\left(  \Sigma\right)  \simeq\Lambda_{\nu}/H$.
\end{proof}

\begin{example} \label{ex}
Let $N$ be the Heisenberg Lie group with Lie algebra $\mathfrak{n}$ spanned by
the adaptable basis $\left\{  Z,Y,X\right\}  $ with non-trivial Lie brackets $\left[  X,Y\right]  =Z.$ 
Let $H$ be a $2$ dimensional commutative Lie group with Lie algebra $\mathfrak{h}=\mathbb{R}A\oplus\mathbb{R}B$ acting on $\mathfrak{n}$ as follows. $\mathbb{R}B=\mathfrak{z}\left(  \mathfrak{g}%
\right)  $ and, $
\left[  A,X\right]  =1/2X,\left[  A,Y\right]  =1/2Y,\left[  A,Z\right]  =Z.$ Applying the procedure above, we obtain 
\begin{enumerate}
\item $\nu=\{1\}$
\item $\Lambda_{\nu}=\{(z,0,0) \in \mathfrak{n}^{\ast} : z\neq 0\}$
\item $\Omega_{\nu}=\{(z,y,x,a,b) \in \mathfrak{g}^{\ast} : z\neq 0,y,x,a,b\in \mathbb{R}\}$
\item $\Sigma=\{(\pm1,0,0,0,b) : b \in \mathbb{R} \}$
\item $\Sigma^{\circ}=\{(\pm1,0,0)\in \mathfrak{n}^{\ast}\}$
\end{enumerate}
\end{example}

\begin{example} \label{ex2}
Let $\mathfrak{g}=\left(\mathbb{R}Z_{1}\oplus\mathbb{R}Z_{2}\oplus
\mathbb{R}Y_{1}\oplus\mathbb{R}Y_{2}\oplus
\mathbb{R}X_{1}\oplus\mathbb{R}X_{2}\right)  \oplus\mathbb{R}A$ with $$\mathfrak{n}=\mathbb{R}Z_{1}\oplus\mathbb{R}Z_{2}\oplus\mathbb{R}
Y_{1}\oplus\mathbb{R}Y_{2}\oplus\mathbb{R}X_{1}\oplus\mathbb{R}
X_{2}$$ and non-trivial Lie brackets%
\begin{align*}
\left[  X_{1}+iX_{2},Y_{1}+iY_{2}\right]    & =Z_{1}+iZ_{2}, \\
\left[  X_{1}-iX_{2},Y_{1}-iY_{2}\right] &=Z_{1}-iZ_{2}\\
\left[  A,X_{1}+iX_{2}\right]    & =\left(  1+i\right)  /2\left(  X_{1}%
+iX_{2}\right)  ,\\
\left[  A,Y_{1}+iY_{2}\right] &=\left(  1+i\right)  /2\left(  Y_{1}%
+iY_{2}\right)  \\
\left[  A,Z_{1}+iZ_{2}\right]    & =\left(  1+i\right)  \left(  Z_{1}%
+iZ_{2}\right)  .
\end{align*}
Then
\begin{enumerate}
\item $\nu=\{1,2\}$
\item $\Lambda_{\nu}=\{(z,\overline{z},0,0,0,0) : z \neq 0\}$
\item $\Omega_{\nu}=\{(z,\overline{z},y,\overline{y},x,\overline{x},a) : z \neq 0,y,x\in\mathbb{C},a\in\mathbb{R}\}$
\item $\Sigma=\{(z,\overline{z},0,0,0,0,0) : z \neq 0\}$
\item $\Sigma^{\circ}=\{(z,\overline{z},0,0,0,0) : z \neq 0\}$
\end{enumerate}
\end{example}

Now, that we have a precise description of the orbital parametrization of the unitary dual of the group, we will take a closer look at the quasiregular representation $\tau$ of $G$ in the next section.

\section{Decomposition of the quasiregular representation}

In this section, we will provide a precise decomposition of $\tau$ as a direct integral of irreducible representations of $G.$ As a result, we will be able to compare the quasiregular representation with the left regular representation of $G$, and to completely settle the question of admissibility for $\tau$

There is a well-known algorithm available for the computation of the
Plancherel measure of $N.$ It is simply obtained by computing the Pfaffian of
a certain skew-symmetric matrix. More precisely, the Plancherel measure on
$\Lambda_{\nu}$ is
\[
d\mu\left(  \lambda\right)  =\left\vert \det\left(  M_{\mathbf{e}^{\circ}%
}\left(  \lambda\right)  \right)  \right\vert ^{1/2} d \lambda= |\mathbf{Pf}%
(\lambda)|d \lambda,
\]
where $M_{\mathbf{e}^{\circ}}\left(  \lambda\right)  =\left(\lambda\left[
Z_{i},Z_{j}\right]\right)_{1\leq i,j\leq\mathbf{e}^{\circ}}.$ In this section, we
will focus on the decomposition of the quasiregular representation
$\tau=\mathrm{Ind}_{H}^{G}1,$ which is a unitary representation of $G$
realized as acting in $L^{2}\left(  N\right)  $ in the following ways,
\begin{align*}
\left(  \tau\left(  n,1\right)  \phi\right)  \left(  m\right)   &
=\phi\left(  n^{-1}m\right) \\
\left(  \tau\left(  1,h\right)  \phi\right)  \left(  m\right)   &  =\left\vert
\delta\left(  h\right)  \right\vert ^{-1/2}\phi\left(  h^{-1}\cdot m\right)
,\text{ with }\delta\left(  h\right)  =\det\left(  Ad\left(  h\right)
\right)  .
\end{align*}
Let $\mathbf{F}$ be the Fourier transform defined on $L^{2}\left(  N\right)
\cap L^{1}\left(  N\right)  ,$ which we extend to $L^{2}\left(  N\right)  .$
Define $$\widehat{\tau}\left(  \cdot\right)  =\mathbf{F}\circ\tau\left(
\cdot\right)  \circ\mathbf{F}^{-1}.$$

\begin{definition}
Let $\lambda\in\Lambda_{\nu}$ a linear functional. A \textbf{polarization
algebra} subordinated to $\lambda$ is a maximal subalgebra of $\mathfrak{n}%
_{\mathbb{C}}$ satisfying the following conditions. Firstly, it is isotropic
for the bilinear form $B_{\lambda}$ defined as $B_{\lambda}(X,Y)=\lambda
[X,Y]$. In other words, it is a maximal subalgebra $\mathfrak{p}$ such that
$\lambda\left(  [\mathfrak{p},\mathfrak{p}]\right)  =0.$ Secondly,
$\mathfrak{p}+\overline{\mathfrak{p}}$ is a subalgebra of $\mathfrak{n}%
_{\mathbb{C}}.$ We will denote a polarization subalgebra subordinated to
$\lambda$ by $\mathfrak{p}(\lambda).$ A polarization is said to be real if
$\mathfrak{p}(\lambda)=\overline{\mathfrak{p}(\lambda)}$. Also, we say that
the polarization $\mathfrak{p}(\lambda)$ is positive at $\lambda$ if $i
\lambda[X,\overline{X}]\geq0$ for all $X\in\mathfrak{p}(\lambda)$.
\end{definition}

Let $\mathbf{e}^{\circ}$ be the set of jump indices corresponding to the
linear functionals in $\Lambda_{\nu}$, and let $\mathbf{e}^{\circ}%
=\frac{\mathbf{d}^{\circ}}{2}$. Referring to Lemma 3.5 in \cite{Didier}, for
any given linear functional $\lambda$, a polarization subalgebra subordinated
to $\lambda$ is given by $\mathfrak{p}(\lambda)=\mathfrak{h}_{\mathbf{d}%
^{\circ}}(\lambda)$. See formula below Equation(\ref{polar}). Unfortunately,
in general the polarization obtained as $\mathfrak{h}_{\mathbf{d}^{\circ}%
}(\lambda)$ is not real and we must in that case proceed by holomorphic
induction in order to construct irreducible representations of $N$. For the
interested reader, a very short introduction to holomorphic induction
is available on page 78 in the book \cite{corwin}.

The following discussion can also be found in \cite{Lispman2} Page $124.$ Given $\lambda\in\Lambda_{\nu},$ let $\pi_{\lambda}$ be an irreducible
representation of $N$ acting in the Hilbert space $\mathcal{H}_{\lambda}$ and
realized via holomorphic induction. Let $\mathcal{X}$ be the domain of
$\mathcal{H}_{\lambda}$ on which the irreducible representation $\pi_{\lambda
}$ is acting on. It is well-known that $\mathcal{X}$ can be identified with
$\mathfrak{n}/\mathfrak{e\times e/d}$, where $$\mathfrak{d=n\cap p}\left(
\lambda\right)  ,\mathfrak{e}=(\mathfrak{p}\left(  \lambda\right)
+\overline{\mathfrak{p}\left(  \lambda\right)  })\cap\mathfrak{n},$$ and
$\mathfrak{p}\left(  \lambda\right)  $ is an $H$-invariant positive
polarization inside $\mathfrak{n}_{\mathbb{C}}$. Finally, $\mathcal{H}%
_{\lambda}=L^{2}\left(  \mathfrak{n}/\mathfrak{e}\right)  \otimes
\mathrm{Hol}\left(  \mathfrak{e/d}\right)  $ with $\mathrm{Hol}\left(
\mathfrak{e/d}\right)  $ denoting the holomorphic functions which are square
integrable with respect to some Gaussian function. It is worth mentioning here
that, if the polarization $\mathfrak{p}(\lambda)$ is real, then $\mathcal{H}%
_{\lambda}=L^{2}(\mathfrak{n}/\mathfrak{e})$, $\mathcal{X}=\mathfrak{n}%
/\mathfrak{e},$ and holomorphic induction here is a just a regular induction.

The choice of how we realize the irreducible representations of $N$ really
depends on the action of the dilation group $H$ on $N$. For example, if the
group $N\rtimes H$ is completely solvable, there is no need to consider the
complexification of $\mathfrak{n}$ since the existence of a positive
polarization always exists for exponential solvable Lie groups. From now on,
we will assume that a convenient choice for a positive polarization subalgebra
has been made for each $\lambda\in\Lambda_{\nu},$ and we denote $\mathcal{H}%
_{\lambda}$ the Hilbert space on which we realize the corresponding
irreducible representation $\pi_{\lambda}$, and $\mathcal{X}$ is a domain on
each we realize the action of $\pi_{\lambda}$. We fix an $H$ quasi-invariant
measure on $\mathcal{X}$, which we denote by $dn,$ and we define
\[
\delta_{\mathcal{X}}\left(  h\right)  =\frac{d\left(  h^{-1}\cdot n\right)
}{dn}.
\]
Furthermore, put $C\left(  h,\lambda\right)  :\mathcal{H}_{\lambda}%
\rightarrow\mathcal{H}_{h\cdot\lambda}$ defined by
\[
C\left(  h,\lambda\right)  f\left(  x\right)  =\left\vert \delta_{\mathcal{X}%
}\left(  h\right)  \right\vert ^{-1/2}f\left(  h^{-1}\cdot x\right)
\]
such that $\pi_{\lambda}\left(  h^{-1}\cdot n\right)  C\left(  h,\lambda
\right)  =C\left(  h,\lambda\right)  \pi_{h\cdot\lambda}\left(  n\right)  $
for all $n\in N.$ We set the following notations. $\Delta$ denotes the modular
function of $G$ where $\Delta(h)= \det(Ad(h)^{-1})$, and $\delta(h)=
\Delta(h)^{-1}$.

\begin{proposition}
Let $\phi\in\mathbf{F}\left(  L^{2}\left(  N\right)  \right)  ,$ we have%
\begin{align*}
\widehat{\tau}_{\lambda}\left(  n\right)  \left(  \mathbf{F}\phi\right)
\left(  \lambda\right)   &  =\pi_{\lambda}\left(  n\right)  \left(
\mathbf{F}\phi\right)  \left(  \lambda\right) \\
\widehat{\tau}_{\lambda}\left(  h\right)  \left(  \mathbf{F}\phi\right)
\left(  \lambda\right)   &  =\left\vert \delta\left(  h\right)  \right\vert
^{1/2}C\left(  h,h^{-1}\cdot\lambda\right)  \left(  \mathbf{F}\phi\right)
\left(  h^{-1}\cdot\lambda\right)  C\left(  h,h^{-1}\cdot\lambda\right)  ^{-1}
.
\end{align*}

\end{proposition}

The proof is elementary. Thus we will omit it. Now, we will describe how to obtain
almost all of the irreducible representations of $G$ via an application of the
Mackey Machine.

\begin{lemma}
If there exists some non zero linear $\lambda\in\Lambda_{\nu},$ and a non
trivial subgroup $K\leq H$ fixing $\lambda,$ then $K$ must fix all elements in $\Lambda_{\nu}.$
\end{lemma}

\begin{proof}
Recall the definition of $\Lambda_{\nu}:$%
\[
\Lambda_{\nu}=\left\{  f\in\Lambda:f\left(  Z_{j}\right)  \neq0,\text{ }%
j\in\left\{  1,2,\cdots,n\right\}  \backslash\mathbf{e}^{\circ}\right\}  .
\]
Suppose there exists a linear functional $f\in\Lambda_{\nu}$ and $h\neq1,$ such
that $h\cdot f=f$. Since the action of $h$ is a diagonal action, then it must
be the case that $ad\log h\left(  Z_{j}\right)  =0$ for all $j\in\left\{
1,2,\cdots,n\right\}  \backslash\mathbf{e}^{\circ}.$ Thus for any $f\in
\Lambda_{\nu},$ we have that
\[
K=\left\{  h\in H:ad\log h\left(  Z_{j}\right)  =0\text{ for }j\in\left\{
1,2,\cdots,n\right\}  \backslash\mathbf{e}^{\circ}\right\}  .
\]
This completes the proof.
\end{proof}

\begin{lemma}\label{claim}
Let $\pi_{\lambda}$ be an irreducible representation of $N$ corresponding to a
linear functional $\lambda\in\Lambda_{\nu}$ via Kirillov's map, and let $K$
the stabilizer subgroup of the coadjoint action of $H$ on $\Lambda_{\nu}.$ We
define the extension of $\pi_{\lambda}$ as $\widetilde{\pi}_{\lambda},$ which
is an irreducible representation of $N\rtimes K$ acting in $\mathcal{H}%
_{\lambda}=L^{2}\left(  \mathfrak{n}/\mathfrak{e}\right)  \otimes
\mathrm{Hol}\left(  \mathfrak{e/d}\right)  $ such that if $\gamma_{\lambda
}\left(  \cdot\right)  $ is the restriction of $C\left(  \lambda,\cdot\right)
$ to $K.$ More precisely, the definition of such extension is given by
$\widetilde{\pi}_{\lambda}\left(  n,k\right)  \phi\left(  x\right)
=\pi_{\lambda}\left(  n\right)  \gamma_{\lambda}\left(  h\right)  \phi\left(
x\right)  .$ Furthermore, let $\left\{  \chi_{\sigma}:\sigma\in\mathfrak{k}%
^{\ast}\right\}  =\widehat{K},$ and recall that $\Sigma^{\circ}$ is the
cross-section for the coadjoint orbits of $H$ in $\Lambda_{\nu}.$ The
following set
\[
\left\{  \mathrm{Ind}_{NK}^{NH}\left(  \widetilde{\pi}_{\lambda}\otimes
\chi_{\sigma}\right)  :\left(  \lambda,\sigma\right)  \in\Sigma^{\circ}%
\times\mathfrak{k}^{\ast}\right\}
\]
exhausts almost all of the irreducible representations of $G$ which will
appear in the Plancherel transform of $G,$ and if $L$ denotes the left regular
representation of $G,$ we have
\[
L\simeq\int_{\Sigma^{\circ}\times\mathfrak{k}^{\ast}}^{\oplus}\mathrm{Ind}%
_{NK}^{NH}\left(  \widetilde{\pi}_{\lambda}\otimes\chi_{\sigma}\right)
\otimes1_{L^{2}\left(  H/K,\mathcal{H}_{\lambda}\right)  }d\mu\left(
\lambda,\sigma\right)
\]
and $d\mu\left(  \lambda,\sigma\right)  $ is absolutely continuous with
respect to the natural Lebesgue measure on $\Sigma^{\circ}\times
\mathfrak{k}^{\ast}.$
\end{lemma}
The claims in Lemma \ref{claim} summarize some standard facts in the analysis of exponential Lie groups. We refer the reader to Theorem 10.2 in \cite{Lipsman 1} where the general case of group extensions is presented, and to \cite{Plancherel} which specializes to the class of groups considered in this paper.
\begin{lemma}
For any $\lambda\in\Lambda_{\nu},$ let $K=\mathrm{Stab}_{G}(\lambda),$ such
that $K\not =\{1\}$. There exists a non trivial representation of $K$ inside
the symplectic group $\mathrm{Sp}\left(  \mathfrak{n/n}(\lambda)\right)  $,
and $\mathfrak{n}(\lambda)$ is the null-space of the matrix $\left(\lambda
[Z_{i},Z_{j}]\right)_{1\leq i,j\leq n}.$
\end{lemma}

\begin{proof}
It is well-known that $\mathfrak{n/n}(\lambda)$ has a smooth symplectic
structure since the bilinear form $B_{\lambda}\left(
X,Y\right)  =\lambda\left[  X,Y\right]  $ is a non degenerate, skew-symmetric
2-form on  $\mathfrak{n/n}(\lambda)$. Let $h\in K,$ since $h\cdot\lambda=\lambda,$ then the bilinear form
$B_{\lambda}\left(  X,Y\right)  $ is $K$-invariant. In other words, for any
$h\in K$, $B_{\lambda}\left(  h\cdot X,h\cdot Y\right)  =B_{\lambda}\left(
X,Y\right)  .$ Thus, there is a natural matrix representation $\beta$ of $K$
such that $\beta\left(  K\right)  $ is a closed subgroup of the symplectic
group $\mathrm{Sp}\left(  \mathfrak{n/n}(\lambda)\right)  .$ Identifying
$\mathfrak{n/n}(\lambda)$ with a supplementary basis of $\mathfrak{n}%
(\lambda)$ in $\mathfrak{n}$, which we denote $\mathcal{B}$, this
representation is nothing but the adjoint representation of $K$ acting on
$\mathcal{B}$.
\end{proof}

\vskip 0.5cm

In this paper, $Z(G)$ stands for the center of the Lie group $G$, and
$\mathfrak{z(g)}$ stands for its Lie algebra. Also, we remind the reader that $\gamma_{\lambda}(\cdot)$ is the restriction of the representation $C(\lambda,\cdot)$ to the group $K$.

\begin{lemma}
Assume that $K_{1}$ is a subgroup of $K$. $\gamma_{\lambda}\left(
K_{1}\right)  =\left\{  1\right\}  $ if and only $K_{1}\leq Z\left(  G\right)
.$
\end{lemma}

\begin{proof}
Clearly if there exists a non trivial subgroup such that $K_{1}\leq Z\left(
G\right)  $ then $\gamma_{\lambda}\left(  K_{1}\right)  =\left\{  1\right\}
.$ For the other way around, let $k\in K_{1}$. Notice that $$\gamma_{\lambda}\left(  k\right)  \phi\left(  x\right)  =\left\vert
\delta_{\mathcal{X}}\left(  h\right)  \right\vert ^{-1/2}\phi\left(
\beta\left(  k\right)  ^{-1}x\right).$$  We have already seen that
$\beta(k)$ is a symplectic matrix, and at least half of its eigenvalues are
$1$. Since for any symplectic matrix, the multiplicity of eigenvalues $1$ if
they occur is even, then it follows that $\beta(k)$ is the identity. Thus, $k$ is a central element.
\end{proof}

\begin{remark}
Let $\beta$ be the finite dimensional representation of $K$ in $\mathrm{Sp}%
(\mathfrak{n}/\mathfrak{n}_{\lambda})$. By the first isomorphism theorem,
$\beta(K)\simeq K/(Z(G)\cap H)$.
\end{remark}

\begin{lemma}
\label{orbit} If there exists some $x\in\mathcal{X}$ with $\phi_{x}%
:K\rightarrow\mathcal{X}$ and $\phi_{x}\left(  k\right)  =k\cdot x$ such that
$\mathrm{rank}\left(  \phi_{x}\right)  =\max_{y\in\mathcal{X}}\left(
\mathrm{rank}\left(  \phi_{y}\right)  \right)  $ then the number of elements
in the cross-section for the $K$ orbit in $\mathcal{X}$ is equal to
$2^{\dim\mathcal{X}}\text{ if }\mathrm{rank}\left(  \phi_{x}\right)
=\dim\mathcal{X},$ and is infinite otherwise.
\end{lemma}

\begin{proof}
Fix a cross-section $\mathcal{C}\simeq\mathcal{X}/K$, for $\mathcal{C}%
\subseteq\mathcal{X}$. For each $x\in\mathcal{C}$, let $r=\max_{x\in
\mathcal{C}}\left(  \mathrm{rank}\left(  \phi_{x}\right)  \right)  $ and
$\mathcal{X}_{1}=\left\{  x\in\mathcal{X}:\mathrm{rank}\left(  \phi
_{x}\right)  =r\right\}  .$ Then, $\mathcal{X}_{1}$ is open and dense in
$\mathcal{X}$. Assume that there exists some $y$ in $\mathcal{C}$ such that
$\mathrm{rank}\left(  \phi_{y}\right)  =\dim\left(  \mathcal{X}\right)  .$ If
$r=\dim\left(  \mathcal{X}\right)  ,$ then $\phi_{y}$ defines a submersion,
which means that $\phi_{y}$ is an open map. Furthermore, $\phi_{y}\left(
K\right)  $ which is the orbit of $y$ is open in $\mathcal{X}_{1}$. From the
definition of the action of $K$ this is only possible if and only if $K$ acts with real eigenvalues, and in that case, the number of orbits is
simply equal to $2^{\dim\mathcal{X}}.$ Now, assume that there exists no $y$ in
$\mathcal{C}$ such that $\mathrm{rank}\left(  \phi_{y}\right)  =\dim
\mathcal{X}$ then the orbits in $\mathcal{X}_{1}$ are always meagre in
$\mathcal{X}_{1}$. So a cross-section will contain an infinite amount of points.
\end{proof}

\begin{lemma}
\label{lemma27} Let $\gamma_{\lambda}\left(  \cdot\right)  $ be the
restriction of $C\left(  \lambda,\cdot\right)  $ to $K.$ We obtain the direct
integral decomposition
\[
\gamma_{\lambda}\simeq\int_{\left(  \mathfrak{k/h\cap z(g)}\right)  ^{\ast}%
}^{\oplus}\chi_{\overline{\sigma}}\otimes1_{\mathbb{C}^{\mathbf{m}}}%
d\overline{\sigma},
\]
where the multiplicity function is uniformly constant, and we have
$\mathbf{m}:\mathfrak{k}^{\ast}\to\mathbb{N}\cup\{\infty\}$ with
$\mathbf{m}(\sigma)$ being equal to the number of elements in the
cross-section $\mathcal{X}/K$.
\end{lemma}

\begin{proof}
Recall that $\gamma_{\lambda}\left(  h\right)  f\left(  x\right)  =\left\vert
\delta_{\mathcal{X}}\left(  h\right)  \right\vert ^{-1/2}f\left(  h^{-1}\cdot
x\right)  $ and let $\mathbf{m}$ be the number of elements in the
cross-section for the $K$-orbits in $\mathcal{X}$. If $K=\left\{  1\right\}  $
then clearly, each point in $\mathcal{X}$ is its own orbit and
$\mathbf{m=\infty.}$ If $K$ acts on some invariant open subset of
$\mathcal{X}$ by spirals, then the cross-section will contain an infinite
number of elements. Let $\mathcal{X}_{1}$ as defined in Lemma \ref{orbit}. We
have the following natural diffeomorphism $\alpha:\mathcal{X}_{1}/K\times
K/\left(  H\cap Z\left(  G\right)  \right)  \rightarrow\mathcal{X}_{1}$ such
that $\alpha\left(  x,\overline{k}\right)  =\overline{k}\cdot x.$ Thus,
$\mathcal{X}_{1}$ becomes a total space with base space $\mathcal{X}_{1}/K,$
and fibers $K/\left(  H\cap Z\left(  G\right)  \right)  \cdot x$ such that
\[
\mathcal{X}_{1}=\displaystyle\bigcup\limits_{x\in\mathcal{X}_{1}/K} \left(
K/\left(  H\cap Z\left(  G\right)  \right)  \cdot x\right)  .
\]
First,\ for each $x$ in the cross-section $\mathcal{X}_{1}/K,$ identify
$K/\left(  H\cap Z\left(  G\right)  \right)  \cdot x$ with $K/\left(  H\cap
Z\left(  G\right)  \right)  ,$ and the Hilbert space $$\mathcal{H}_{\lambda}
\simeq\ \left(  L^{2}\left(  K/\left(  H\cap Z\left(  G\right)  \right)
\right)  \right)  ^{\mathbf{m}}\simeq L^{2}\left(  K/\left(  H\cap Z\left(
G\right)  \right)  \right)  \otimes\mathbb{C}^{\mathbf{m}}.$$ In fact for each
linear functional $\lambda$, the representation $\gamma_{\lambda}$ can be
modelled as being quasi-equivalent to the left regular representation on
$K/\left(  H\cap Z\left(  G\right)  \right)  .$ Let $\phi$ be a function in
$\mathcal{H}_{\lambda}$ and for each $x\in\mathcal{X}_{1}/K,$ we define
$\phi_{x}$ as the restriction of the function $\phi$ to the orbit of $x.$ It
is easy to see that the action of $\gamma_{\lambda}\left(  \cdot\right)  $
becomes just a left translation acting on $\phi_{x}$ for each $x\in
\mathcal{X}_{1}/K.$ $K/\left(  H\cap Z\left(  G\right)  \right)  $ being a
commutative Lie group, we can decompose its left regular representation by
using its group Fourier transform. Letting $\left(  \mathfrak{k/h\cap
z(g)}\right)  ^{\ast},$ the unitary dual of the group $K/\left(  H\cap
Z\left(  G\right)  \right)  ,$ we obtain a decomposition of the representation
$\gamma_{\lambda}$ into its irreducible components as follows.
\[
\gamma_{\lambda}\simeq\int_{\left(  \mathfrak{k/h\cap z(g)}\right)  ^{\ast}%
}^{\oplus}\chi_{\overline{\sigma}}\otimes1_{\mathbb{C}^{\mathbf{m}}}%
d\overline{\sigma},
\]
where $\chi_{\overline{\sigma}}$ are characters defined on $Z(G)\cap H$ and
$\int_{\left(  \mathfrak{k/h\cap z(g)}\right)  ^{\ast}}^{\oplus}%
\chi_{\overline{\sigma}}\otimes1_{\mathbb{C}^{\mathbf{m}}}d\overline{\sigma}$
is modelled as acting in the Hilbert space $\int_{\left(  \mathfrak{k/h\cap
z(g)}\right)  ^{\ast}}^{\oplus}\mathbb{C} \otimes1_{\mathbb{C}^{\mathbf{m}}%
}d\sigma.$ This completes the proof.
\end{proof}

\begin{lemma}
\label{plancherel decomp} Let $\Lambda_{\nu}\rightarrow\Sigma^{\circ}%
\simeq\Lambda_{\nu}/H$ be the quotient map induced by the action of $H$. The
push-forward of the Lebesgue measure on $\Lambda_{\nu}$ via the quotient map
is a measure equivalent to the Lebesgue measure on $\Sigma^{\circ}%
\simeq\Lambda_{\nu}/H.$
\end{lemma}

\begin{proof}
This Lemma follows from the following facts. The quotient map is a submersion
everywhere, and the push-forward of a Lebesgue measure via a submersion is
equivalent to a Lebesgue measure on the image set.
\end{proof}

Now, we will compute an explicit decomposition of the Plancherel measure on
$\Lambda_{\nu}$ under the action of the dilation group $H$. We first recall
the more general theorem for disintegration of Borel measures.

\begin{lemma}
Let $G$ be a locally compact group. Let $X$ be a left Borel $G$-space and
$\mu$ a quasi-invariant $\sigma$-finite positive Borel measure on $X.$ Assume
that there is a $\mu$-null set $X_{0\text{ }}$such that $X_{0\text{ }}$ is
$G$-invariant and $X-X_{0}$ is standard. Then for all $x\in X-X_{0},$ the
orbit $G\cdot x$ is Borel isomorphic to $G/G_{x}$ under the natural mapping,
and there is a quasi-invariant measure $\mu_{x}$ concentrated on the orbit
$G\cdot x$ such that for all $f\in L^{1}\left(  X,\mu\right)  ,$%
\[
\int_{X}f\left(  x\right)  d\mu\left(  x\right)  =\int_{\left(  X-X_{0}%
\right)  /G}\int_{G/G_{x}}f\left(  g\cdot x\right)  d\mu_{x}\left(
gG_{x}\right)  d\overline{\mu}\left(  x\right)  ,
\]
where $G_{x}$ is the stability group at $x.$
\end{lemma}
We refer the interested reader to \cite{Lipsman 1} for a proof of the above lemma.
\begin{proposition}
\label{planchereldecomp} (Disintegration of the Plancherel measure) Under the
action of $H$ the Plancherel measure on $\Lambda_{\nu}$ is decomposed into a
measure on the cross-section $\Sigma^{\circ}$ and a family of measures on each
orbit such that for any non negative measurable function $F\in L^{1}%
(\Lambda_{\nu})$, we have
\[
\int_{\Lambda_{\nu}}F\left(  f\right)  \left\vert \mathbf{Pf}\left(  f\right)
\right\vert df=\int_{\Sigma^{\circ}}\int_{H/K}F\left(  \overline{h}\cdot
\sigma\right)  d\omega_{\sigma}\left(  \overline{h}\right)  \left\vert
\mathbf{Pf}\left(  \sigma\right)  \right\vert d\sigma
\]
where for each $\sigma\in\Sigma^{\circ},$ $d\omega_{\sigma}\left(
\overline{h}\right)  =\Delta\left(  \overline{h}\right)  d\overline{h},$
$d\overline{h}$ is the natural Haar measure on $H/K,$ and $d\sigma$ is a
Lebesgue measure on $\Sigma^{\circ}$ with $\overline{h}=hK$, and $\Delta$ is
the modular function defined on the group $H/K.$
\end{proposition}
The proof is obtained via some elementary computations involving changing variables. It is quite trivial. Thus, we shall omit it.
\begin{theorem}
\label{decomposition of quasi regular representation} The quasiregular
representation is unitarily equivalent to the following direct integral
decomposition:
\[
\int_{\Sigma^{\circ}}^{\oplus}\left(  \int_{\left(  \mathfrak{k/z\left(
g\right)  }\mathfrak{\cap h}\right)  ^{\ast}}^{\oplus}\mathrm{Ind}_{NK}%
^{NH}\left(  \widehat{\pi}_{\lambda}\otimes\chi_{\overline{\sigma}}\right)
\otimes1_{\mathbb{C}^{\mathbf{m}}}d\overline{\sigma}\right)  \left\vert
\mathbf{Pf}\left(  \lambda\right)  \right\vert d\lambda,
\]
with multiplicity function $\mathbf{m}$ equal to $2^{\dim\mathcal{X}}\text{ if
}\mathrm{rank}\left(  \phi_{x}\right)  =\dim\mathcal{X},$ or infinite otherwise.
\end{theorem}

The proof for Theorem \ref{decomposition of quasi regular representation}
follows from Lemma \ref{plancherel decomp}, Lemma \ref{lemma27}, and Theorem
\ref{Ronald} which is proved in Theorem 7.1 \cite{Lipsman}.

\begin{proposition}
The quasiregular representation $\tau=\mathrm{Ind}_{H}^{G}\left(  1\right)  $
is contained in the left regular representation if and only if $\dim(Z\left(
G\right)  \cap H)=0$.
\end{proposition}

\begin{proof}
Assume that $Z\left(  G\right)  \cap H$ is not equal to the trivial group
$\{1\}.$ We have proved that $\gamma_{\lambda}\simeq\int_{\left(
\mathfrak{k}/(\mathfrak{z}\left(  \mathfrak{g}\right)  \cap\mathfrak{h}%
\right)  )^{\ast}}^{\oplus}\chi_{\sigma}\otimes1_{\mathbb{C}^{\mathbf{m}%
\left(  \sigma\right)  }}d\overline{\sigma}.$ By Proposition
\ref{plancherel decomp} and also, Theorem 3.1 in \cite{Lispman2}, we have
\[
\tau\simeq\int_{\Sigma^{\circ}}^{\oplus}\int_{\left(  \mathfrak{k}%
/(\mathfrak{z}\left(  \mathfrak{g}\right)  \cap\mathfrak{h}\right)  )^{\ast}%
}^{\oplus}\mathrm{Ind}_{NK}^{NH}\left(  \widetilde{\pi}_{\lambda}\otimes
\chi_{\sigma}\right)  \otimes1_{\mathbb{C}^{\mathbf{m}\left(  \lambda
,\sigma\right)  }}d\overline{\sigma} d\lambda.
\]
The measure $d\overline{\sigma}$ is a measure belonging to the Lebesgue class
measure on $\left(  \mathfrak{k}/\mathfrak{z}\left(  \mathfrak{g}\right)
\cap\mathfrak{h}\right)  ^{\ast}$, which we identify with $\mathbb{R}%
^{\dim(\mathfrak{k}/(\mathfrak{z}\left(  \mathfrak{g}\right)  \cap
\mathfrak{h})}$. The Plancherel measure of the group $G$ is supported on
$\Sigma^{\circ}\times\mathfrak{k}^{\ast}$ and belongs to the Lebesgue class
measure $d\lambda d\sigma$ such that $d\sigma$ is the Lebesgue measure on
$\mathfrak{k}^{\ast}=\mathbb{R}^{\dim(\mathfrak{k})}$. Clearly, if
$\dim(Z\left(  G\right)  \cap H)>0$, then $\mathbb{R}^{\dim(\mathfrak{k}%
/(\mathfrak{z}\left(  \mathfrak{g}\right)  \cap\mathfrak{h})}$ is meagre in
$\mathbb{R}^{\dim(\mathfrak{k})}$. Thus, the measure occurring in the
decomposition of the quasiregular representation, and the measure occurring in
the decomposition of the left regular representation are mutually singular if
and only if $\dim(Z\left(  G\right)  \cap H)>0$. Finally, we have
\begin{align*}
L  &  \simeq\int_{\Sigma^{\circ}}^{\oplus}\int_{\mathfrak{k}^{\ast}}^{\oplus
}\mathrm{Ind}_{NK}^{NH}\left(  \widetilde{\pi}_{\lambda}\otimes\chi_{\sigma
}\right)  \otimes1_{L^{2}\left(  H/K,\mathcal{H}_{\lambda}\right)  }\text{
}d\sigma d\lambda\\
&  \simeq\int_{\Sigma^{\circ}}^{\oplus}\int_{\mathfrak{k}^{\ast}}^{\oplus
}\mathrm{Ind}_{NK}^{NH}\left(  \widetilde{\pi}_{\lambda}\otimes\chi_{\sigma
}\right)  \otimes1_{\mathbb{C}^{\infty}}\text{ }d\sigma d\lambda.
\end{align*}
Since the irreducible representations occurring in the decomposition of $L$
have uniform infinite multiplicities, the quasiregular representation
$\tau=\mathrm{Ind}_{H}^{G}1$ is contained in the left regular representation
if and only if $\dim(Z\left(  G\right)  \cap H)=0$.
\end{proof}

Finally we have our main result.

\begin{theorem}
Assume that $G=N\rtimes H$ is unimodular. Then $\tau$ is never admissible.
Assume that $G$ is nonunimodular. $\tau$ is admissible if and only if
$\dim\left(  Z\left(  G\right)  \cap H\right)  =0.$
\end{theorem}

\begin{proof}
First, assume that $G$ is unimodular. Clearly if $$\dim\left(  Z\left(
G\right)  \cap H\right)  =0$$ then, $\tau$ will be contained in the left
regular representation. However $G$ being unimodular, it is known (see
\cite{Fuhr}) that any subrepresentation of the left regular representation is
admissible if and only if
\begin{equation}
\label{integral}\int_{\Sigma}\mathbf{m}\left(  \lambda,\sigma\right)
d\mu(\lambda,\sigma) <\infty.
\end{equation}
However that is not possible because, the multiplicity is constant a.e.,
$\mathbf{m}\left(  \lambda,\sigma\right)  =\mathbf{m}$%
\begin{align*}
\int_{\Sigma}\mathbf{m}\left(  \lambda,\sigma\right)  d\mu\left(
\lambda,\sigma\right)   &  =\int_{\Sigma}\mathbf{m\cdot}\text{ }d\mu\left(
\lambda,\sigma\right) \\
&  =\mathbf{m\cdot}\text{ }\mu\left(  \Sigma\right)  .
\end{align*}
Now for the first case. Assume that $\mathbf{m}$ is infinite, then clearly,
the integral will diverge. For the second case, assume that $\mathbf{m}$ is
finite. Then, there exists at least a non trivial $k\in\mathfrak{k}$ such that
$\Sigma=\Sigma^{\circ}\times\mathfrak{k}^{\ast}$ and, using Currey's measure(\cite{Plancherel}), up to multiplication by a constant,
\[
d\mu\left(  \lambda,\sigma\right)  =\left\vert \mathbf{Pf}_{\mathbf{e}}\left(
\lambda,\sigma\right)  \right\vert d\lambda d\sigma.
\]
where $
\mathbf{Pf}_{\mathbf{e}}\left(  \lambda,\sigma\right)  =\det\left(\left(  \lambda,\sigma\right)  \left[  Z_{i_{r}},Z_{j_{s}}\right]
\right)_{1\leq r,s\leq \mathbf{d}}.$ It is thus clear from the definition of the action of $H$, that the function
$\mathbf{Pf}_{\mathbf{e}}\left(  \lambda,\sigma\right)  $ is really a function
of $\lambda.$ Thus, we just write $\mathbf{Pf}_{\mathbf{e}}\left(
\lambda,\sigma\right)  =\mathbf{Pf}_{\mathbf{e}}\left(  \lambda\right)  $ and
\begin{align*}
\int_{\Sigma}\mathbf{m}\left(  \lambda,\sigma\right)  d\mu\left(
\lambda,\sigma\right)   &  =\mathbf{m}\int_{\Sigma^{\circ}}\int_{\mathfrak{k}%
^{\ast}}\text{ }\left\vert \mathbf{Pf}_{\mathbf{e}}\left(  \lambda\right)
\right\vert d\lambda d\sigma\\
&  =\mathbf{\infty.}%
\end{align*}
If $G$ is unimodular and $\dim\left(  Z\left(  G\right)  \cap H\right)  >0$
then, $\tau$ must be disjoint from the left regular representation. Now assume
that $G$ is nonunimodular. We have 2 different cases. If $\dim\left(  Z\left(
G\right)  \cap H\right)  >0$ then the quasiregular representation is disjoint
from the left regular representation which automatically prevents $\tau$ from
being admissible. Secondly, assume that $\dim\left(  Z\left(  G\right)  \cap
H\right)  =0.$ We have
\[
\tau\simeq\int_{\Sigma^{\circ}}^{\oplus}\int_{\mathfrak{k}^{\ast}}^{\oplus
}\mathrm{Ind}_{NK}^{NH}\left(  \widetilde{\pi}_{\lambda}\otimes\chi_{\sigma
}\right)  \otimes1_{\mathbb{C}^{\mathbf{m}\left(  \lambda,\sigma\right)  }%
}d\sigma d\lambda,
\]
and of course, as seen previously, the multiplicity function is uniformly
constant and, $\mathbf{m}\left(  \lambda,\sigma\right)  \leq\infty.$ Thus,
$\tau$ is quasi-equivalent with the left regular representation. $G$ being
nonunimodular, it follows that $\tau$ is admissible.
\end{proof}

\begin{remark}
We bring the attention of the reader to the fact that the theorem above
supports Conjecture 3.7 in \cite{CV} which states that a monomial
representation of a unimodular exponential solvable Lie group $G$ never has
admissible vectors. The general case remains an open problem.
\end{remark}
Based on our main theorem, we can assert the following.
\begin{remark}
 Let $N$ be a nilpotent Lie group with Lie algebra $\mathfrak{n}.$ Let $H$ be given such that at least one of the basis element of $\mathfrak{h}$ commutes with all basis elements of $\mathfrak{n}.$ Then $Z(N\rtimes H)\cap H$ is clearly non trivial, and $\tau$ cannot be admissible as a representation of $G$.  \end{remark}

\section{Examples}
In this section, we will present several examples, and we will show how to apply our results in order to settle the admissibility of $\tau$ in each case. 

\begin{example} Coming back to Example \ref{ex}, clearly $G$ is not unimodular. Since the center of the group has a non-trivial intersection with $H$ then $\tau$ is not an admissible representation. \end{example}

\begin{example} Recall Example \ref{ex2}. Since $G$ is nonunimodular and since the center of the group is trivial, then $\tau$ is an admissible representation of $G.$
\end{example}
\begin{example}
Let $G$ a Lie group with Lie algebra $\mathfrak{g}$ spanned by $\left\{
Z,Y,X,A_{1},A_{2},A_{3}\right\}  $ such that
\begin{align*}
\left[  X,Y\right]    & =Z,\left[  A_{1},X\right]  =X,\\
\left[  A_{2},X\right]    & =X,\left[  A_{3},X\right]  =2X\\
\left[  A_{1},Y\right]    & =Y,\left[  A_{2},Y\right]  =-Y,\\
\left[  A_{3},Y\right]    & =-Y,\left[  A_{1},Z\right]  =Z,\\
\left[  A_{3},Z\right]    & =Z.
\end{align*}
Since the center of $G$ is equal to
\[
\exp\left(
\mathbb{R}
\left(  -\frac{1}{2}A_{1}-\frac{3}{2}A_{2}+A_{3}\right)  \right)  <H
\]
then $\tau$ is not admissible.
\end{example}

\begin{example}
Let $G$ a Lie group with Lie algebra $\mathfrak{g}$ spanned by $$\left\{
Z,Y,X,W,A_{1},A_{2},A_{3},A_{4}\right\}  $$ with non-trivial Lie brackets
\begin{align*}
\left[  X,Y\right]    & =Z,\left[  W,X\right]  =Y,\\
\left[  A_{1},W\right]    & =\frac{1}{3}W,\left[  A_{1},X\right]  =\frac{1}%
{3}X,\\
\left[  A_{1},Y\right]    & =\frac{2}{3}Y,\left[  A_{1},Z\right]  =Z\\
\left[  A_{2},W\right]    & =-W,\left[  A_{2},X\right]  =X,\\
\left[  A_{2},Z\right]    & =Z,\left[  A_{3},W\right]  =1/5W,\\
\left[  A_{3},X\right]    & =2/5X,\left[  A_{2},Y\right]  =3/5Y,\\
\left[  A_{3},Z\right]    & =Z,\left[  A_{4},X\right]  =1/2X,\\
\left[  A_{4},Y\right]    & =1/2Y,\left[  A_{4},Z\right]  =Z.
\end{align*}
In this example the Lie algebra $\mathfrak{h}$ is spanned by the vectors $A_1,A_2,A_3,A_4.$ The center of $G$ is equal to%
\[
\exp\left(
\mathbb{R}
\left(  -\frac{9}{10}A_{1}-\frac{1}{10}A_{2}-A_{3}\right)  \right)
\exp\left(
\mathbb{R}
\left(  -\frac{3}{4}A_{1}-\frac{1}{4}A_{2}+A_{4}\right)  \right)  <H
\]
then $\tau$ is not admissible
\end{example}
\begin{example} Let us suppose that $\mathfrak{g}$ is spanned by the vectors $$U_1,U_2,Z_1,Z_2,Z_3,X_1,X_2,X_3,A$$ and $\mathfrak{h}$ is spanned by the vector $A$. Furthermore, assume that we have the following non-trivial Lie brackets $$[X_3,X_2]=Z_1, [X_3,X_1]=Z_2, [X_2,X_1]=Z_3,[A,U_1+i U_2]=(1+i)(U_1+i U_2).$$ We remark that in this example, the nilradical of $\mathfrak{g}$ is a step-two freely generated nilpotent Lie algebra with $3$ generators. Since $G$ is nonunimodular, and since the center of $G$ is trivial, then $\tau$ is admissible.
\end{example}

\begin{example} 
Let $N$ be the Heisenberg group%
\[
N=\left\{  \left(
\begin{array}
[c]{cccc}%
1 & x & y & z\\
0 & 1 & 0 & y\\
0 & 0 & 1 & 0\\
0 & 0 & 0 & 1
\end{array}
\right)  :\left(
\begin{array}
[c]{c}%
z\\
y\\
x
\end{array}
\right)  \in%
\mathbb{R}
^{3}\right\}  ,
\]
and the dilation group $H$ is isomorphic to $%
\mathbb{R}
^{2}$ such that
\[
H=\left\{  \left(
\begin{array}
[c]{cccc}%
e^{t} & 0 & 0 & 0\\
0 & e^{t-r} & 0 & 0\\
0 & 0 & e^{r} & 0\\
0 & 0 & 0 & 1
\end{array}
\right)  :\left(
\begin{array}
[c]{c}%
t\\
r
\end{array}
\right)  \in%
\mathbb{R}
^{2}\right\}  .
\]
The action of $H$ on $N$ is given as follows.
\begin{align*}
& \left(
\begin{array}
[c]{cccc}%
e^{t} & 0 & 0 & 0\\
0 & e^{t-r} & 0 & 0\\
0 & 0 & e^{r} & 0\\
0 & 0 & 0 & 1
\end{array}
\right)  \left(
\begin{array}
[c]{cccc}%
1 & x & y & z\\
0 & 1 & 0 & y\\
0 & 0 & 1 & 0\\
0 & 0 & 0 & 1
\end{array}
\right)  \left(
\begin{array}
[c]{cccc}%
e^{t} & 0 & 0 & 0\\
0 & e^{t-r} & 0 & 0\\
0 & 0 & e^{r} & 0\\
0 & 0 & 0 & 1
\end{array}
\right)  ^{-1}\\
& =\left(
\begin{array}
[c]{cccc}%
1 & xe^{r} & ye^{t}e^{-r} & ze^{t}\\
0 & 1 & 0 & ye^{t-r}\\
0 & 0 & 1 & 0\\
0 & 0 & 0 & 1
\end{array}
\right)  \allowbreak
\end{align*}
It is easy to see that the Lie algebra of $G$ is spanned by $\left\{  Z,Y,X,A\right\}
$ with non-trivial Lie brackets%
\begin{align*}
\left[  A_{1},Z\right]    & =Z,\left[  A_{1},Y\right]  =Y\\
\left[  A_{2},Y\right]    & =-Y,\left[  A_{2},X\right]  =X.
\end{align*}
Here
\[
K=\left\{  \left(
\begin{array}
[c]{cccc}%
1 & 0 & 0 & 0\\
0 & e^{-r} & 0 & 0\\
0 & 0 & e^{r} & 0\\
0 & 0 & 0 & 1
\end{array}
\right)  :r\in%
\mathbb{R}
\right\}
\]
but the center of the $G$ is trivial. Thus, there is a non-trivial subgroup of the dilation group stabilizing the center of $N$
and thus stabilizing almost all of elements of the unitary dual of $N.$ The
spectrum of the left regular representation of $G=N\rtimes H$ is supported on
the two disjoint lines, and the irreducible representations occurring in
decomposition of the left regular representation occur with infinite
multiplicities. Also, the spectrum of the quasiregular representation $\tau$
is parametrized by two disjoint lines, but the irreducible representations
occurring in the decomposition of $\tau$ occur twice almost everywhere. Since
the group $G$ is nonunimodular, and $\tau$ is contained in $L$ then $\tau$ is admissible. 
\end{example}

\begin{example}
Let us suppose that $\mathfrak{n}$ is spanned by $T_{1},T_{2},Z,Y,X$ such that
$\left[  X,Y\right]  =Z,$ $\mathfrak{h}$ is spanned by $A_{1},$ $A_{2}%
,A_{3},A_{4},A_{5}$ such that
\begin{align*}
\left[  A_{2},X\right]    & =1/2X,\\
\left[  A_{2},Y\right]    & =1/2Y,\\
\left[  A_{2},Z\right]    & =Z,\\
\left[  A_{3},X\right]    & =X\\
\left[  A_{3},Y\right]    & =-Y\\
\left[  A_{5},X\right]    & =X,\left[  A_{6},Y\right]  =Y\\
\left[  A_{3},T_{1}+iT_{1}\right]    & =\left(  1+i\right)  \left(
T_{1}+iT_{1}\right)  \\
\left[  A_{4},T_{1}+iT_{1}\right]    & =\left(  2+2i\right)  \left(
T_{1}+iT_{1}\right)  \\
\left[  A_{1},T_{1}+iT_{1}\right]    & =\left(  1+i\right)  \left(
T_{1}+iT_{1}\right).
\end{align*}
The center of $G$ is given by%
\[
\exp\left(
\mathbb{R}
\left(  A_{1}-2A_{2}-1/2A_{4}\right)  \right)  \exp\left(
\mathbb{R}
\left(  A_{3}-1/2A_{4}-A_{5}\right)  \right)  <H.
\]
Thus $\tau$ is not admissible.
\end{example}

 The author is grateful to the anonymous referee for his careful reading, comments, corrections, and helpful suggestions

\end{document}